\documentclass[11pt]{amsart}
\usepackage{amsmath,amssymb,amsfonts,amsthm,amstext}

\usepackage{url,xspace}
\usepackage{color,graphicx}
\usepackage{graphics}
\usepackage{tikz}
\usepackage{enumerate}

\usepackage{verbatim} %adds comment environment
\usepackage[hidelinks]{hyperref} %adds hyperlinks
\usepackage{a4}
\usepackage{amssymb}
\newcommand{\old}[1]{}

\newcommand{\R}{{\mathbb R}}
\newcommand{\N}{{\mathbb N}}

\renewcommand{\P}{{\mathbb P}}

\newcommand{\mudiag}{{\mu^{\nearrow}}}

\newcommand{\isd}{\stackrel{d}{=}}
\newcommand{\tod}{\stackrel{d}{\to}}

\newcommand{\cP}{\mathcal{P}}

\newcommand{\cPtwo}{\mathcal{P}^{(2)}}

\newcommand{\mutwo}{\mu^{(2)}}
\newcommand{\nutwo}{\nu^{(2)}}
\newcommand{\Ttwo}{T^{(2)}}

\newtheorem{theorem}{Theorem}
\newtheorem{lemma}[theorem]{Lemma}
\newtheorem{proposition}[theorem]{Proposition}
\newtheorem{corollary}[theorem]{Corollary}

\theoremstyle{definition}
\newtheorem{example}{Example}
\newtheorem{remark}[example]{Remark}

\numberwithin{example}{section}
\numberwithin{figure}{section}
\numberwithin{theorem}{section}
\numberwithin{equation}{section}

\newcounter{mycount}

\title{Minimax functions on Galton-Watson trees}
\author{James B.\  Martin and Roman Stasi{\'n}ski}
\date{20 June 2018}



\begin{document}

\begin{abstract}
We consider the behaviour of minimax recursions defined on random 
trees. Such recursions give the value of a 
general class of two-player combinatorial games. 
We examine in particular the case where the tree is given by a Galton-Watson branching process, truncated at some depth $2n$, and the terminal values of the level-$2n$ nodes are drawn independently from some 
common distribution. The case of a regular tree was previously 
considered by Pearl, who showed that as $n\to\infty$ the 
value of the game converges to a constant, and by Ali Khan, Devroye
and Neininger, who obtained a distributional limit under a suitable
rescaling. 

For a general offspring distribution, there is a 
surprisingly rich variety of behaviour: the (unrescaled) value of the game may converge to a constant, or to a discrete limit with several atoms, or to a continuous distribution. We also give distributional 
limits under suitable rescalings in various cases. 

We also address questions of \textit{endogeny}. Suppose the game is played on a tree with many levels, so that the terminal values are far from the root. To be confident of playing a good first move, do we need to see the whole tree and its terminal values, or can we play close to optimally by inspecting just the first few levels of the tree? The answers again depend in an interesting way on the offspring distribution.

We also mention several open questions.
\end{abstract}

\maketitle

\section{Introduction}
In this paper we consider the behaviour of minimax recursions
defined on random trees. 

Consider a finite rooted tree with depth $m$. We will call the root ``level 0", the children of the root ``level 1", and so on. Suppose every node at levels $0,1,\dots, m-1$ has at least one child; the nodes at level $m$ are all leaves. 
Suppose every leaf node (i.e.\ every node at level $m$) has
some real value associated to it. Then 
recursively propagate the values towards the root
in a minimax way: each node at an odd level gets 
a value which is the max of the values of its children,
and each node at an even level gets a value which is
the min of the values of its children.

This minimax procedure has a natural interpretation
in terms of a two player game. Two players alternate turns; 
a token starts at the root, and a move of the game consists
of moving the token from its current node to one of
the children of that node. The leaf nodes are terminal
positions; the outome of the game is the value associated
to the leaf node where the game ends. Player 1 is 
trying to minimise this outcome, and player 2 is trying
to maximise it. The outcome of the game with ``optimal play"
is the value associated to the root. 

Suppose the terminal values are random, drawn independently from some common distribution.

Pearl \cite{Pearl} considered the case where the tree is regular (every non-leaf node has $d$ children for some $d\geq 2$) and the terminal values are independent and uniformly distributed 
on the interval $[0,1]$. 
For simplicity assume that the depth of the tree is even; 
write $W_{2n}$ for a random variable representing
the value at the root of a tree of depth $2n$. 
Pearl showed that $W_{2n}$ converges in distribution 
to a constant as $n\to\infty$. This result 
was refined by Ali Khan, Devroye and Neininger
\cite{ADN}, who derived an asymptotic distribution
for $W_{2n}$ after appropriate rescaling. 

In this paper we consider the case where the tree is given by a Galton-Watson branching process, truncated at level $2n$. This generalisation leads to a surpsingly rich variety of behaviour,
depending on the offspring distribution of the branching process. 
For example, the limiting distribution of $W_{2n}$ 
may be concentrated at a single point (as in the regular case), or may now have several atoms, or may even be continuous. 

There is also a rich interplay between the two sources of randomness now present in the model (the tree itself, and the terminal values at the leaves). Suppose we play the game on a tree with many levels, so that the terminal values are far from the root. In order to be confident of playing a good first move, do we need to see the whole tree and terminal values, or can we play close to optimally by inspecting just the structure of the first few levels of the tree? Such questions can be formulated precisely in terms of the 
\textit{endogeny} property for certain recursive
tree processes, as introduced by \cite{AldBan}.
The answers again depend in an interesting way on the offspring distribution. 

Such questions concerning the relative importance 
of local tree structure and terminal values are of
considerable interest in understanding the 
effectiveness of certain tree-search algorithms such 
as \textit{Monte Carlo tree search} (MCTS) -- see
\cite{MCTSsurvey} for a survey. MCTS has famously 
been applied in recent years to games such as go,
where it provided a considerable increase in playing
strength \cite{MCTSgo} even before being allied with powerful deep learning techniques \cite{alphago}.
For some games, simple versions of these algorithms,
without local evaluation functions, and with
only very crude input from the terminal values 
(given for example by ``random rollouts" 
through unexplored parts of the tree),
are nonetheless able to converge quickly towards
good lines of play. Understanding which
aspects of a game's structure make
such convergence possible is an interesting challenge both
in theoretical and in practical terms.

Our main resuts concerning distributional limits are presented
in the next section. In Section \ref{sec:examples} we discuss a range of examples and mention some open problems. The results about endogeny are given in Section \ref{sec:endogeny}.
The main proofs are given in Section \ref{sec:distribution_proofs}
and Section \ref{sec:endogeny_proof}.

Before that we mention some recent related work. 
Broutin, Devroye and Fraiman \cite{BroutinDevroyeFraiman}
consider recursive distributional equations 
(including those of minimax type) defined on 
Galton-Watson trees conditioned to have a given 
total size $n$. Holroyd and Martin \cite{HolroydMartin}
consider minimax-type games (and various mis{\`e}re and
asymmetric variants) defined on (perhaps infinite)
Galton-Watson trees, with particular emphasis
on the nature of 
phase transitions for the outcomes of the game as the
underlying offspring distribution varies
(see Section \ref{sec:examples} for
further comments). Note that in both 
\cite{BroutinDevroyeFraiman} and \cite{HolroydMartin},
unlike in the case of this paper, the offspring
distribution puts positive weight at $0$, 
so that there are leaves close to the root. 

Similar questions arise in the context of 
random \textsc{AND/OR} trees and random Boolean functions. 
For example the model of Pemantle and Ward
\cite{Pemantle} involves a regular tree in which 
each node independently is a max or a min with 
equal probability; see Section \ref{subsec:comments_identity}
for comments on the relation to a particular case of our model. See for example Broutin and Mailler \cite{BroutinMailler} for a variety of recent results in a more general setting, and many relevant references. 

\subsection{Main results}
Consider a Galton-Watson tree with an offspring distribution with mass function
$p_1, p_2, p_3,\dots$ on $\{1,2,3,\dots\}$ (note that every individual has at least
one child). Let $G(x)=\sum_{k=0}^\infty p_k x^k$ be the probability generating
function of the offspring distribution (which is a strictly increasing function 
mapping $[0,1]$ to $[0,1]$ bijectively). We will also write
throughout
\begin{gather}
\label{Rdef}
R(x)=1-G(x)
\\
\intertext{and}
\label{fdef}
f(x)=R(R(x)). 
\end{gather}

Truncate the tree at level $2n$, so that all the
vertices at level $2n$ are leaves. Let the 
terminal values associated to the leaves be i.i.d.\
uniform on $[0,1]$ (independently of the structure of the
tree). Recursively, assign values to the internal nodes
of the tree (in particular, to the root) using the minimax
procedure defined above. See Figure \ref{fig:minimax-example} for an illustration. 

\begin{figure}
\begin{center}
\tikzset{
node_level/.style={draw=none, fill=none}
}
\begin{tikzpicture}
[scale=0.7, transform shape, 
level distance=10mm,
every node/.style={circle, draw, fill=red!10},
level 1/.style={sibling distance=66mm, nodes={fill=blue!10}},
level 2/.style={sibling distance=35mm, nodes={fill=red!10}},
level 3/.style={sibling distance=16mm, nodes={fill=blue!10}},
level 4/.style={sibling distance=7mm, nodes={rectangle, fill=white}}]
\node[label=right:{$W_4$}] (n){$\wedge$}
  child {node (n1) {$\vee$}
  	child {node[label=right:{$W_2^{(1,1)}$}] (n11) {$\wedge$}
  		child {node (n111) {$\vee$}
			child {node (n1111) {$U_1$}
				}	
  		}
      	child {node (n112) {$\vee$}
	      	child {node (n1121) {$U_2$}}
	      	child {node (n1122) {$U_3$}} 	      		      	
      	}
    }
    child {node[label=right:{$W_2^{(1,2)}$}] (n12) {$\wedge$}
  		child {node (n121) {$\vee$}
	      	child {node (n1211) {$U_4$}} 
	      	child {node (n1212) {$U_5$}} 	      	  		
  		}
      	child {node (n122) {$\vee$}
	      	child {node (n1221) {$U_6$}} 
	      	child {node (n1222) {$U_7$}} 
	      	child {node (n1223) {$U_8$}} 	      		      	      	
      	}
    }  
  }
  child {node (n2) {$\vee$}
  	child {node[label=right:{$W_2^{(2,1)}$}] (n21) {$\wedge$}
  		child {node (n211) {$\vee$}
	      	child {node (n2111) {$U_9$}}   		
  		}
      	child {node (n212) {$\vee$}
	      	child {node (n2121) {$U_{10}$}}       	
      	}
    }
    child {node[label=right:{$W_2^{(2,2)}$}] (n22) {$\wedge$}
  		child {node (n221) {$\vee$}
	      	child {node (n2211) {$U_{11}$}} 
	      	child {node (n2212) {$U_{12}$}} 	      	  		
  		}
      	child {node (n222) {$\vee$}
	      	child {node (n2221) {$U_{13}$}}       	
      	}
      	child {node (n223) {$\vee$}
	      	child {node (n2231) {$U_{14}$}} 
	      	child {node (n2232) {$U_{15}$}} 
	      	child {node (n2233) {$U_{16}$}} 	      		      	      	
      	}
    }  
  }
  child {node (n3) {$\vee$}
    child {node[label=right:{$W_2^{(3,1)}$}] (n31) {$\wedge$}
        child {node (n311) {$\vee$}
	      	child {node (n3111) {$U_{17}$}}      
        }
        child {node (n312) {$\vee$}
	      	child {node (n3121) {$U_{18}$}}
	      	child {node (n3122) {$U_{19}$}}
	      	child {node (n3123) {$U_{20}$}
				child [grow=right] {node [node_level] (level0) {{ }level 4} 
				edge from parent[draw=none]
            	child [grow=up] {node [node_level] (level1) {{ }level 3} 
            	edge from parent[draw=none]
				child [grow=up] {node [node_level] (level2) {{ }level 2} 
				edge from parent[draw=none]
				child [grow=up] {node [node_level] (level3) {{ }level 1} 
				edge from parent[draw=none]
				child [grow=up] {node [node_level] (level4) {{ }level 0} 
				edge from parent[draw=none]
			}}}}}}	      		      	      
      }
    }
  };
\end{tikzpicture}
\caption{
\label{fig:minimax-example}
An example of a minimax tree, with $4$ levels. 
Here all non-leaf nodes have 1, 2 or 3 children.}
\end{center}
\end{figure}
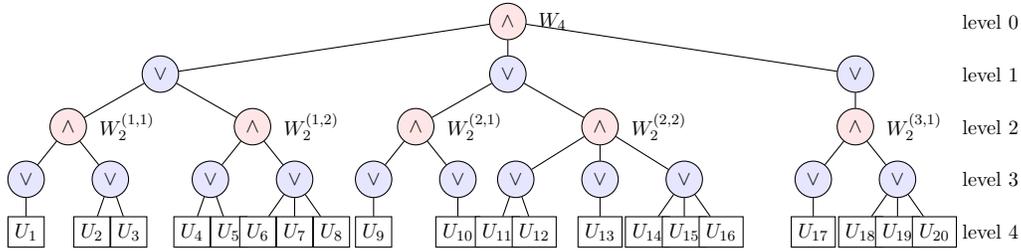

(Note that there is a nothing particularly special about 
uniform boundary conditions. By a simple rescaling we can 
map between this case and the case of i.i.d.\ boundary values
from any other continuous distribution. Later we will also consider discrete boundary values, for example those taking values only 0 and 1, where we can interpret 0 as a win for the first player, and 1 as a win for the second player).

We denote by $W_{2n}$ the random variable associated with the root of a tree of depth $2n$. The we have a distributional recursion:
\begin{equation}\label{eq:rde}
W_{2n}\isd 
\min_{1\leq i\leq M}
\max_{1\leq j\leq M_i}
W^{(i,j)}_{2n-2},
\end{equation}
where $M$ and $M_1,M_2,M_3,\dots$ are i.i.d.\ draws from the offspring distribution,
and $W^{(i,j)}_{2n-2}$, $i,j\in\N$, are i.i.d.\ copies of the random variable
$W_{2n-2}$, independent of $M$ and $\{M_i\}$. 

Then a simple generating function computation (see 
the beginning of 
Section \ref{sec:distribution_proofs}) gives
\begin{equation}
\P\left(W_{2n}\leq x\right)=f\big(\P(W_{2n-2}\leq x)\big),
\label{quantilerecursion}
\end{equation}
where $f$ is defined at (\ref{fdef}).
So to look at the behaviour of the $W_{2n}$ for large $n$,
we will be interested in the function $f$ and in particular
its fixed points.

We begin with the results  
for the case of a regular tree. 

\begin{theorem}
\label{thm:regular}
Suppose $p_d=1$ for some $d\geq 2$.
\begin{itemize}
\item[(a)]
(Pearl \cite{Pearl})
\[
W_{2n}\tod w \text{ as } n\to\infty,
\]
where $q$ is the unique fixed point in $(0,1)$ 
of the function $f_{d\operatorname{-reg}}$ defined by 
\begin{equation}\label{freg}
f_{d\operatorname{-reg}}(x)=1-\left(1-x^d\right)^d.
\end{equation}
\item[(b)]
(Ali Khan, Devroye and Neininger\cite{ADN})
Let $\xi=f_{d\operatorname{-reg}}'(q)$. Then
\[
\xi^{n}\left(W_{2n}-q\right)\tod W \text{ as } n\to\infty,
\]
where $W$ has a continuous distribution function $F_W$ 
which satisfies
$F_W(x)=f_{d\operatorname{-reg}}(F_W(x/\xi))$. 
\end{itemize}
\end{theorem}

Now we will consider general offspring distributions. 
Since $G$ is increasing and bijective as a function from $[0,1]$
to $[0,1]$, we have that $R=1-G$ is decreasing and bijective.
and $f=R\circ R$ is again increasing and bijective.
Also $G$ is analytic on $[0,1)$, so that $f$ is analytic
on $(0,1)$. 

We'll be particularly interested in fixed points of the function $f$. The function $R$ itself has a single fixed point, which is obviously also a fixed point of $f$. 
Otherwise the fixed points of $f$ come in pairs:
if $q$ is one then so is $R(q)$. One such pair 
are the points 0 and 1.
We will say that a fixed point $q$ of $f$ 
is \textit{unstable from the right} if $q<1$ and
$\displaystyle\lim_{\epsilon\to0}\lim_{n\to\infty}f^n(q+\epsilon)
>q$; similarly \textit{unstable from the left}
if $q>0$ and $\displaystyle\lim_{\epsilon\to0}\lim_{n\to\infty}f^n(q-\epsilon)<q$.

For a regular tree, Theorem \ref{thm:regular}
tells us that the distribution of $W_{2n}$ converges to a constant.
For general distributions, we still have convergence in distribution,
but now we may have a ``genuinely random outcome" in the limit
as the tree becomes large; the limiting distribution may 
have more than one atom (and in some surprising cases, the 
distribution of $W_{2n}$ can simply be the same uniform distribution for all $n$). 

\begin{theorem}
\label{thm:unscaled}
$W_{2n}\tod W$ as $n\to\infty$, for some random variable $W$.
There are two cases. 
\begin{itemize}
\item[(a)] If $f$ is the identity function,
then $W_{2n}\sim U[0,1]$ for all $n$. 
\item[(b)] Otherwise, 
let $Q$ be the set of fixed points of $f$ which are 
unstable from at least one side. 

Then $W$ is discrete and has atoms precisely at the elements 
of $Q$.

For $q\in Q$, define
%\begin{align}
%%\nonumber
%%q_- {} & = \begin{cases} \textrm{supremum of fixed points of } f \textrm{ smaller than }q, & \textrm{if $q>0$ and $q$ is unstable from the left}\\ q & \textrm{otherwise}
%%	\end{cases} \\
%%\label{qminusdef}
%%q_+ {} & = \begin{cases} \textrm{infimum of fixed points of } f \textrm{ larger than }q, & \textrm{if $q<1$ and $q$ is unstable from the right}\\ q & \textrm{otherwise}
%%	\end{cases}.
%%\end{align}
%or
%\begin{align}
%\nonumber
%q_- {} & = \begin{cases} {\displaystyle\lim_{\epsilon\to0}\lim_{n\to\infty}f^n(q-\epsilon)}, & \textrm{if $q>0$ and $q$ is unstable from the left}\\ q & \textrm{otherwise}
%	\end{cases} \\
%\label{qminusdef}
%q_+ {} & = \begin{cases} {\displaystyle\lim_{\epsilon\to0}\lim_{n\to\infty}f^n(q+\epsilon)}, & \textrm{if $q<1$ and $q$ is unstable from the right}\\ q & \textrm{otherwise}
%	\end{cases}.
%\end{align}
%or
\begin{align}
\nonumber
q_- {} & = \begin{cases} 
\sup\{x: x<q, x=f(x)\},
& \textrm{if $q>0$ and $q$ is unstable from the left}\\ q & \textrm{otherwise}
	\end{cases} \\
\label{qminusdef}
q_+ {} & = \begin{cases}
\inf\{x: x>q, x=f(x)\},
& \textrm{if $q<1$ and $q$ is unstable from the right}\\ q & \textrm{otherwise}
	\end{cases}.
\end{align}
Then $\P(W=q)=q_+-q_-$.
\end{itemize}
\end{theorem}

It's not hard to show that $x\in Q$ if and only if $R(x)\in Q$. Hence again the atoms of the distributional limit $W$ come in pairs, with the possible exception of the fixed point of $R$. In Section 
\ref{subsec:endpoints}, we comment in particular on the case where $W$ has atoms at $0$ and $1$.

For $q\in(0,1)$, we may write  (\ref{qminusdef}) alternatively as
$q_-=\displaystyle\lim_{\epsilon\to0}\lim_{n\to\infty}f^n(q-\epsilon)$
and $q_+=\displaystyle\lim_{\epsilon\to0}\lim_{n\to\infty}f^n(q+\epsilon)$ (this follows straightforwardly from the monotonicity 
and continuity of $f$). 

In the next results we consider fluctuations around the atoms 
of the limiting distributions obtained in Theorem \ref{thm:unscaled}(b). The appropriate rescaling around a point $q\in Q$ depends 
on the derivative of $\xi=f'(q)$. If $q\in Q$ 
then we must have $\xi\geq 1$.

%Since $R=1-G$, we have
%\begin{align}\label{dRdx}
%R_2'(x)=\frac{d}{dx}R(R(x))&=R'(R(x))R'(x)\\
%&=G'(1-G(x))G'(x).\nonumber
%\end{align}

\begin{theorem}
\label{thm:rescaled}
Consider the model defined by (\ref{eq:rde}). Assume that $f$ is not the identity function and let $Q$ be the set of fixed points of $f$ unstable from at least one side. 

Let $q\in Q$. Define $q_-$ and $q_+$ as at (\ref{qminusdef}), and
let $\xi=f'(q)$. Then:
\begin{enumerate}[(a)]
\item If $1 < \xi < \infty$, then
\begin{align*}
\mathcal{L} \left(\xi^n(W_{2n} - q)  \ | \ W_{2n} \in [q_-, q_+] \right) \xrightarrow[]{} \mathcal{L}(V) \textrm{ as } n \rightarrow \infty,
\end{align*}
where $V$ is a random variable with a continuous distribution function.
\label{thm:main_greater_t_1}
\item Suppose $\xi = 1$, and $k\geq 2$ is such that 
$f^{(r)}(q)=0$ for $1<r<k$ and $f^{(k)}(q)\ne 0$.
Then
\begin{align*}
\mathcal{L} \left(n^{\frac{1}{k-1}}(W_{2n}-q)  \ | \ W_{2n} \in [q_-, q_+] \right) \xrightarrow{} \mathcal{L}(V),
\end{align*}
where for $a = \left( \frac{k(k-2)!}{f^{(k)}(q)} \right)^{\frac{1}{k-1}}$ we have $V=\begin{cases}a & \textrm{w.p. }\frac{q_+-q}{q_+-q_-}\\
-a & \textrm{w.p. }\frac{q-q_-}{q_+-q_-}\end{cases}$.
\label{thm:main_eq_1}
\item If $\xi = \infty$, then $q \in \{0,1\}$. Assume now that
\begin{align}
\mathbb{E}(M \mathbb{I}_{M \leq n}) = \sum_{k=1}^n k p_k \sim c n^\rho \text{ as } n \to \infty
\label{assumption:mean}
\end{align}
for some $c>0$ and $\rho \in (0,1)$, where $M$ is distributed
according to the offspring distribution of the Galton-Watson tree, and let $K= \min\{i : p_i \neq 0\}$. Then $K<1/(1-\rho)$, and 
$|f(t)-q| \sim C|t-q|^{K(1-\rho)}$ as $t \to q$ for some $C > 0$. Moreover, 
\begin{align*}
\mathcal{L} (-[K(1- \rho)]^n \log |W_{2n}-q|  \ | \ W_{2n} \in [q_-, q_+] ) \xrightarrow[]{} \mathcal{L}(Y),
\end{align*}
where $Y$ is a random variable such that $\P(Y \in (0, \infty)) = 1$. 
\label{thm:main_eq_inf}
\end{enumerate}
\label{prof:prop_main}
\end{theorem}

The scaling limits in part (a) are the closest ones to 
the result for the regular tree from Theorem \ref{thm:regular}. 
Note that when $q$ is an endpoint of the interval,
the limiting distribution $V$ is now one-sided, supported
on $(0,\infty)$ when $q=0$ and on $(-\infty,0)$ when $q=1$.

For part (b), recall that $f$ is analytic on $(0,1)$ so certainly
if $q\in(0,1)$, such a $k$ exists. Conceivably, there might be no such $k$ in some cases where $q=0$ or $q=1$ (although we know of 
no example where analyticity fails at $0$ or $1$ except when the derivative is infinite). 

On the other hand, many cases with $\xi=\infty$ are not covered by part (c). It seems challenging to describe all possible asymptotics; 
however, the assumption (\ref{assumption:mean}) is satisfied for an important class of power-law distributions with infinite mean, satisfying $\P(X > x) \sim x^{1-\alpha}$ with $\alpha \in (1,2)$. 

\section{Examples, discussion and open questions}
\label{sec:examples}
Our final main results, concerning the endogeny property, will be stated in Section \ref{sec:endogeny}. Before 
that, we discuss a variety of examples illustrating
the results of Theorems \ref{thm:unscaled} and \ref{thm:rescaled}.

First consider a case where each node has $1$ or $3$ children. 
This simple family already displays an interesting range of behaviours. 
Let $p_1=p$ and $p_3=1-p$, for $p\in[0,1]$.
In Figure \ref{figure:ternary}, we plot the function 
$f(x)-x$ for $x\in[0,1]$, for a variety of values of $p$. 
Fixed points of $f$ correspond to zeros
of the curve. A crossing from negative to positive corresponds
to an unstable fixed point. 

\begin{figure}[h]
\begin{center}
\includegraphics[width=0.32\textwidth]{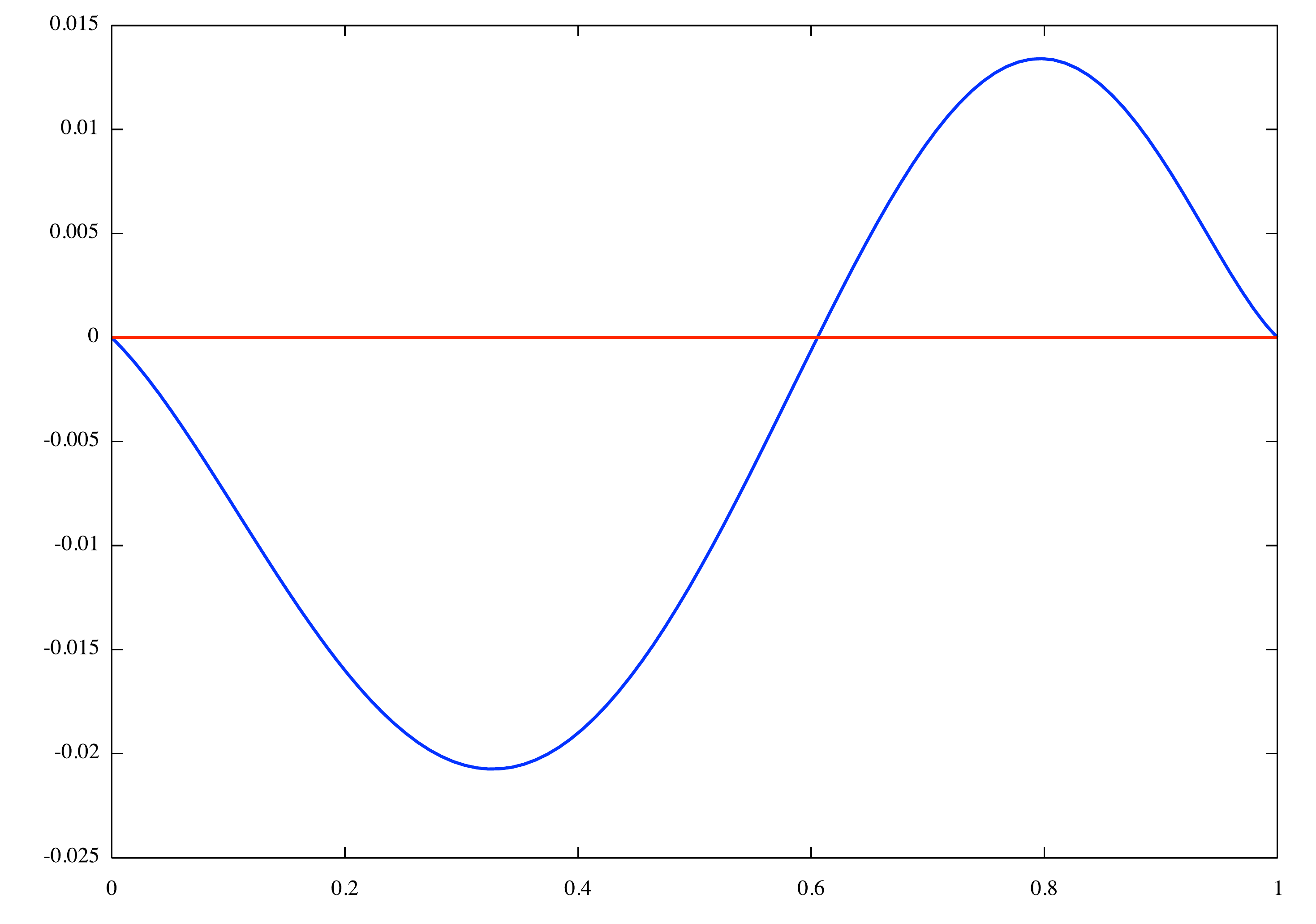}
\includegraphics[width=0.32\textwidth]{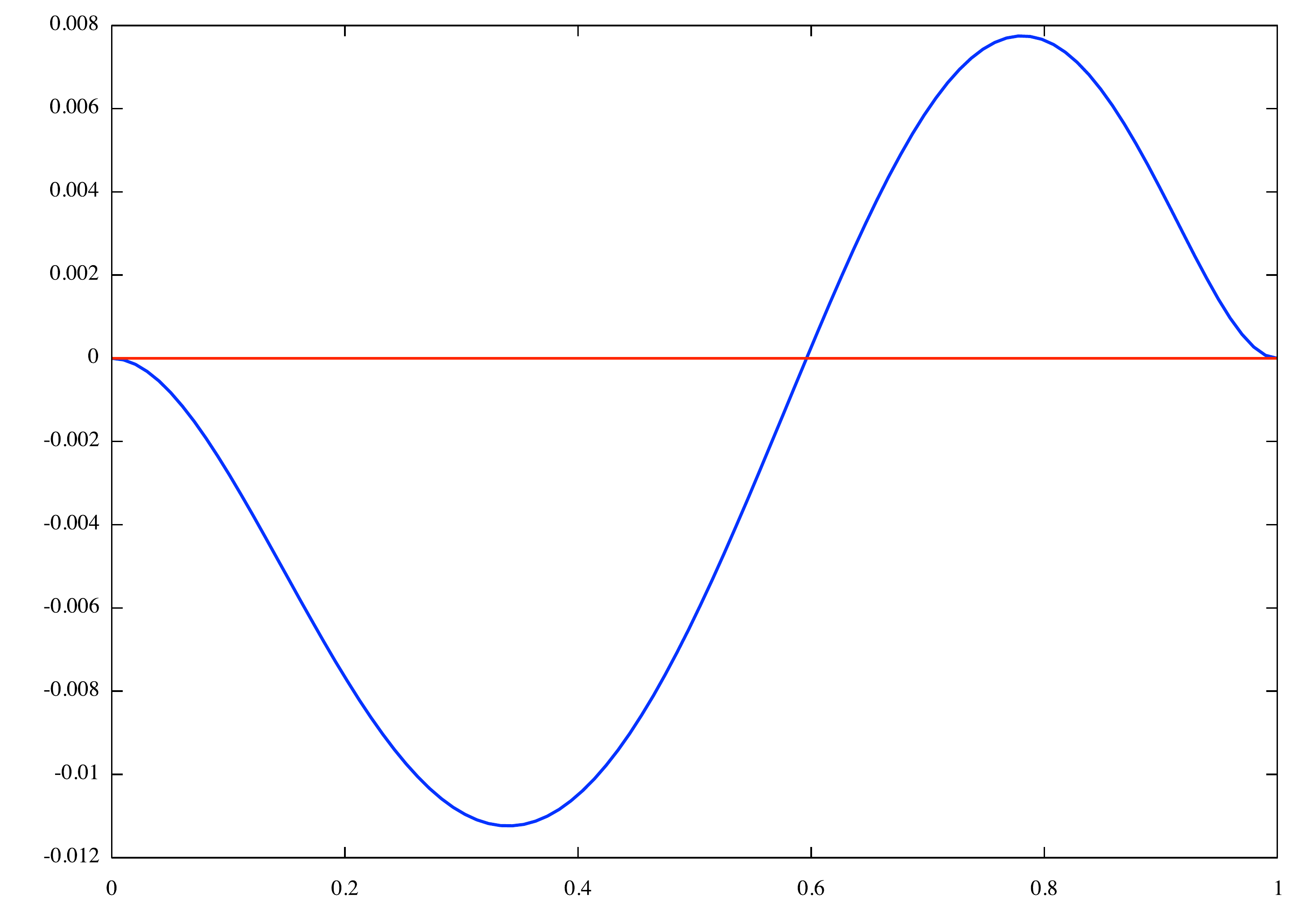}
\includegraphics[width=0.32\textwidth]{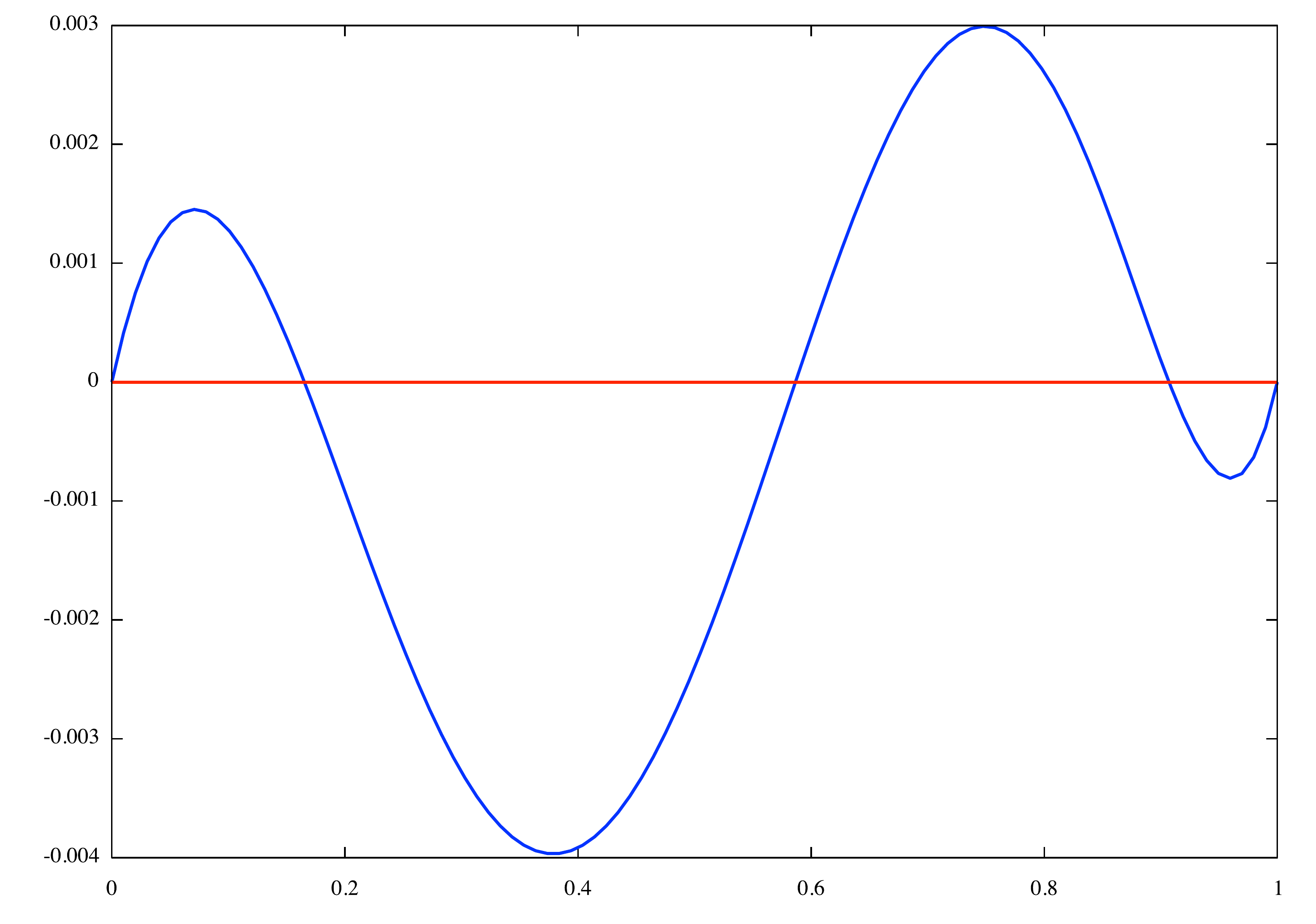}
\includegraphics[width=0.32\textwidth]{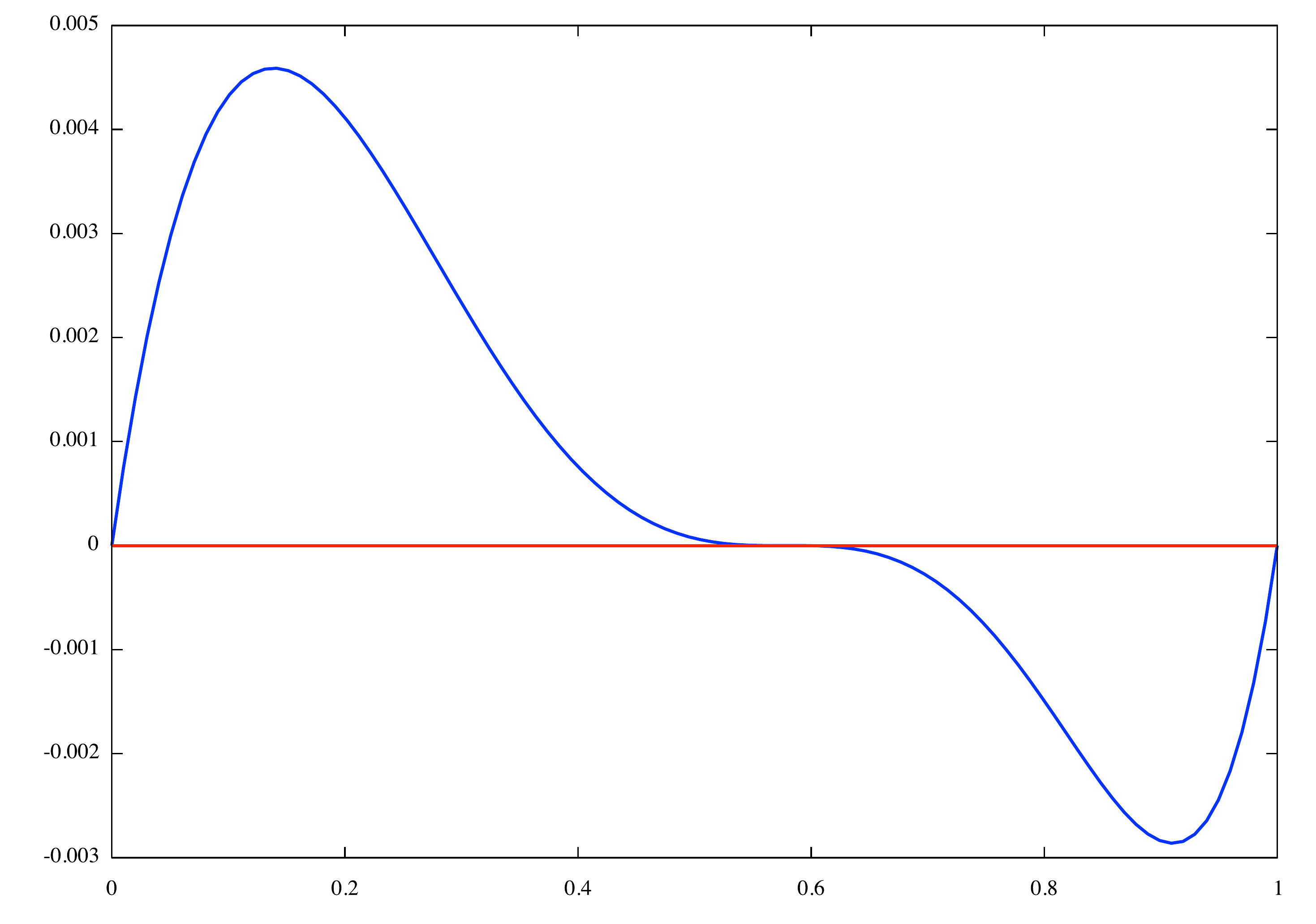}
\includegraphics[width=0.32\textwidth]{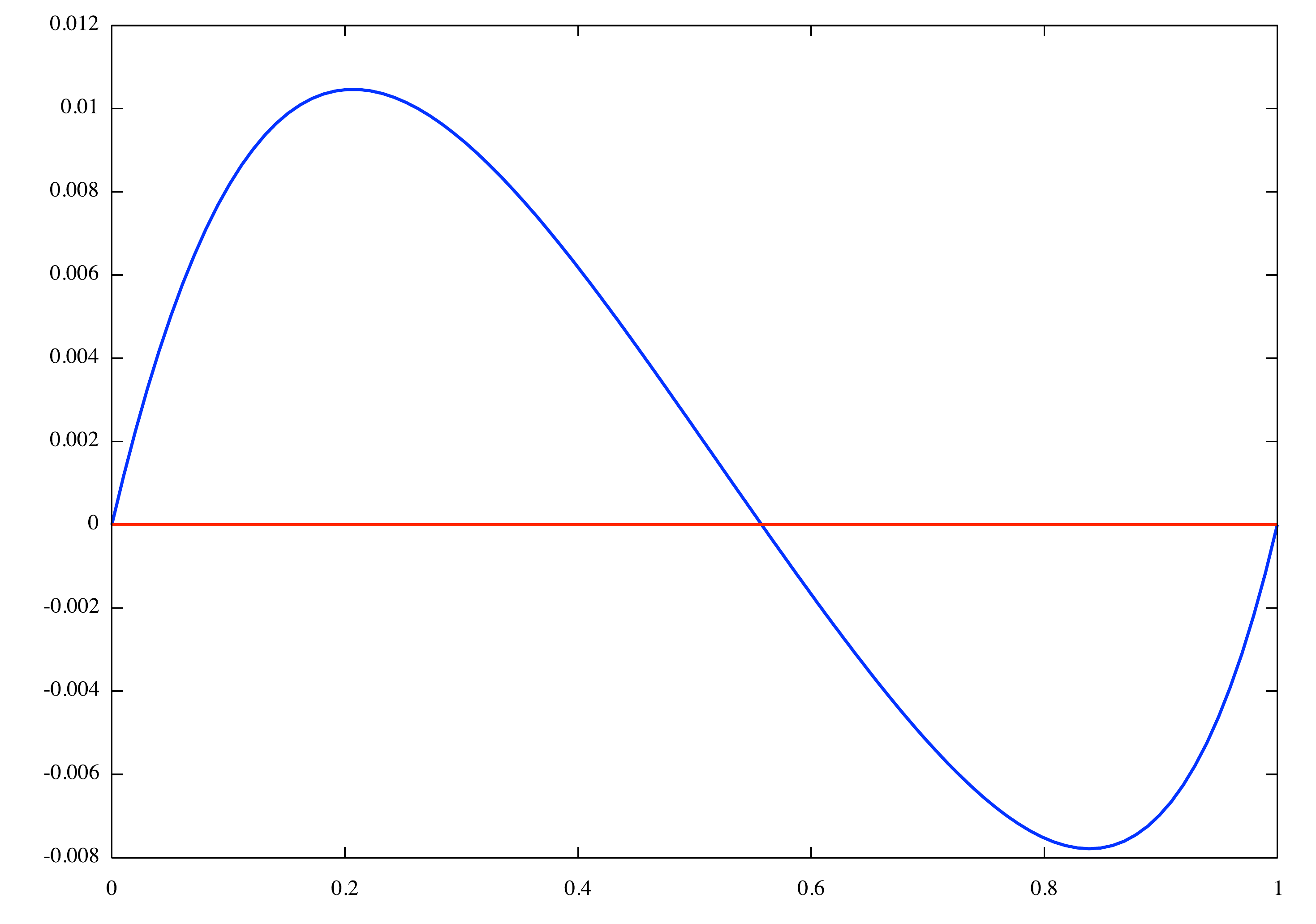}
\end{center}
\caption{
\label{figure:ternary}
The function $f(x)-x$ for $x\in[0,1]$ for the family of distributions with $p_1=p$ and $p_3=1-p$, with (a) $p=0.45$, (b) $p=0.5$, (c) $p=0.55$, (d) $p=0.598$, and (e) $p=0.7$.}
\end{figure}
When $p<0.5$, the points 0 and 1 are stable and there
is a unique unstable fixed point in $(0,1)$, just as in the case of a regular tree; $W_{2n}$ converges to a constant. At $p=0.5$, we have $f'(0)=f'(1)=1$; the slope of $f'(x)-x$ at 0 and 1 is 0, but the points are still stable. For $p>0.5$, the points 0 and 1 are unstable, and the limiting distribution $W$ in Theorem \ref{thm:unscaled} puts positive mass at 0 and 1. At first, there is also positive mass at another fixed point in $(0,1)$. However above a critical point at roughly $p=0.598$, two of the fixed points disappear, leaving only a stable fixed point in $(0,1)$, and the limiting distribution is concentrated only on the points 0 and 1.

Some further illustrative examples are shown in Figure 
\ref{figure:critical}. 
\begin{figure}[h]
\begin{center}
\includegraphics[width=0.32\textwidth]{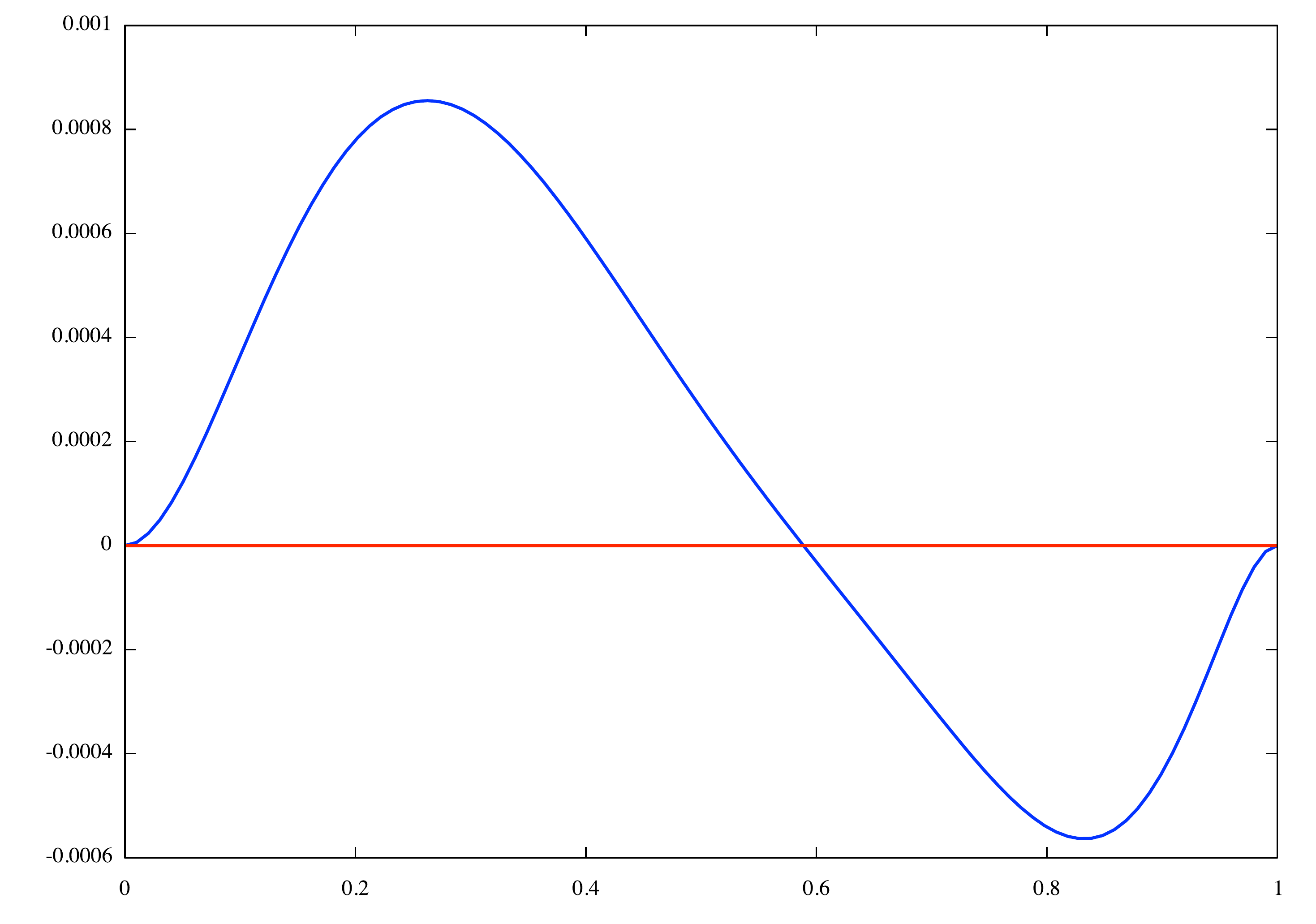}
\includegraphics[width=0.32\textwidth]{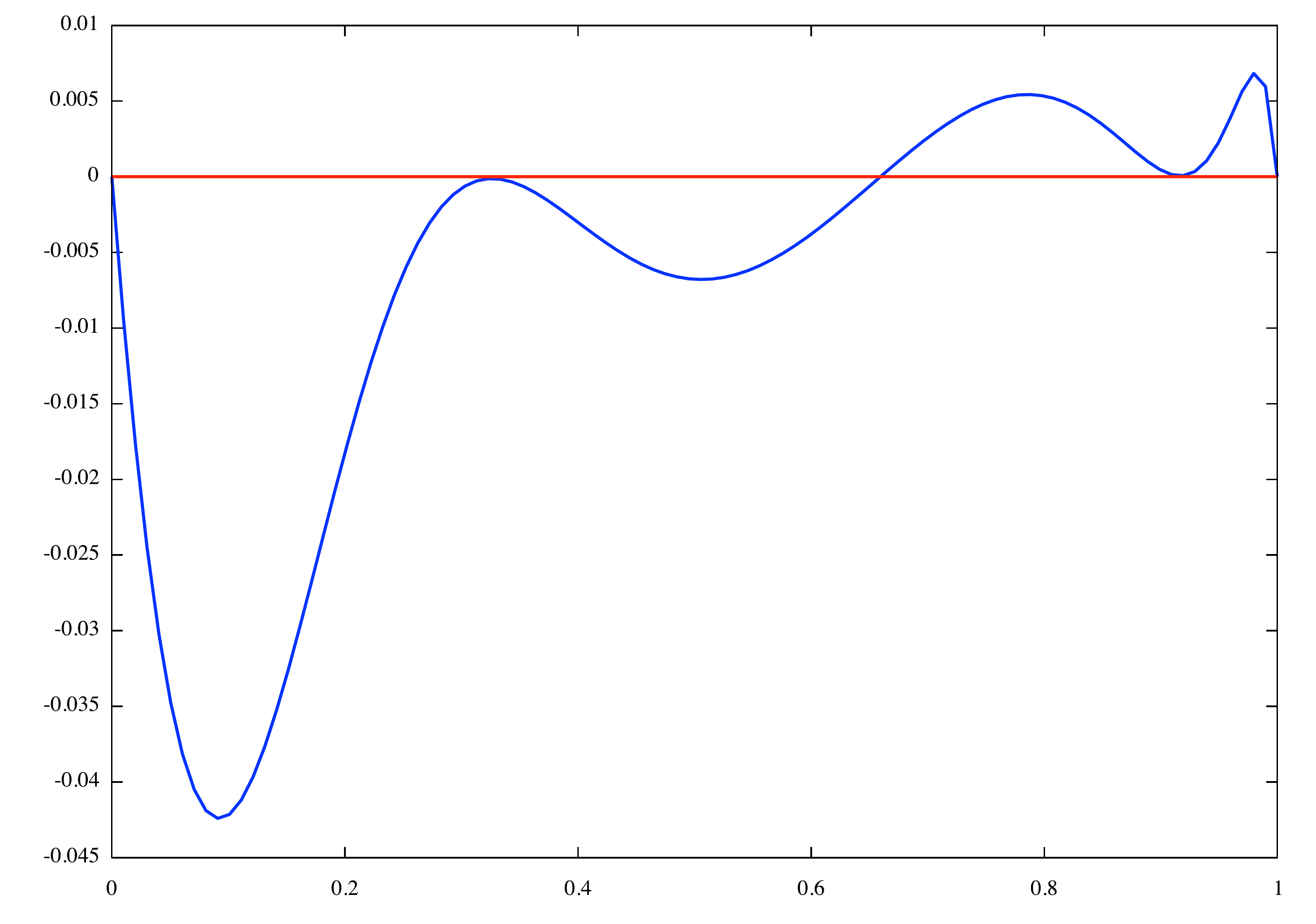}
\includegraphics[width=0.32\textwidth]{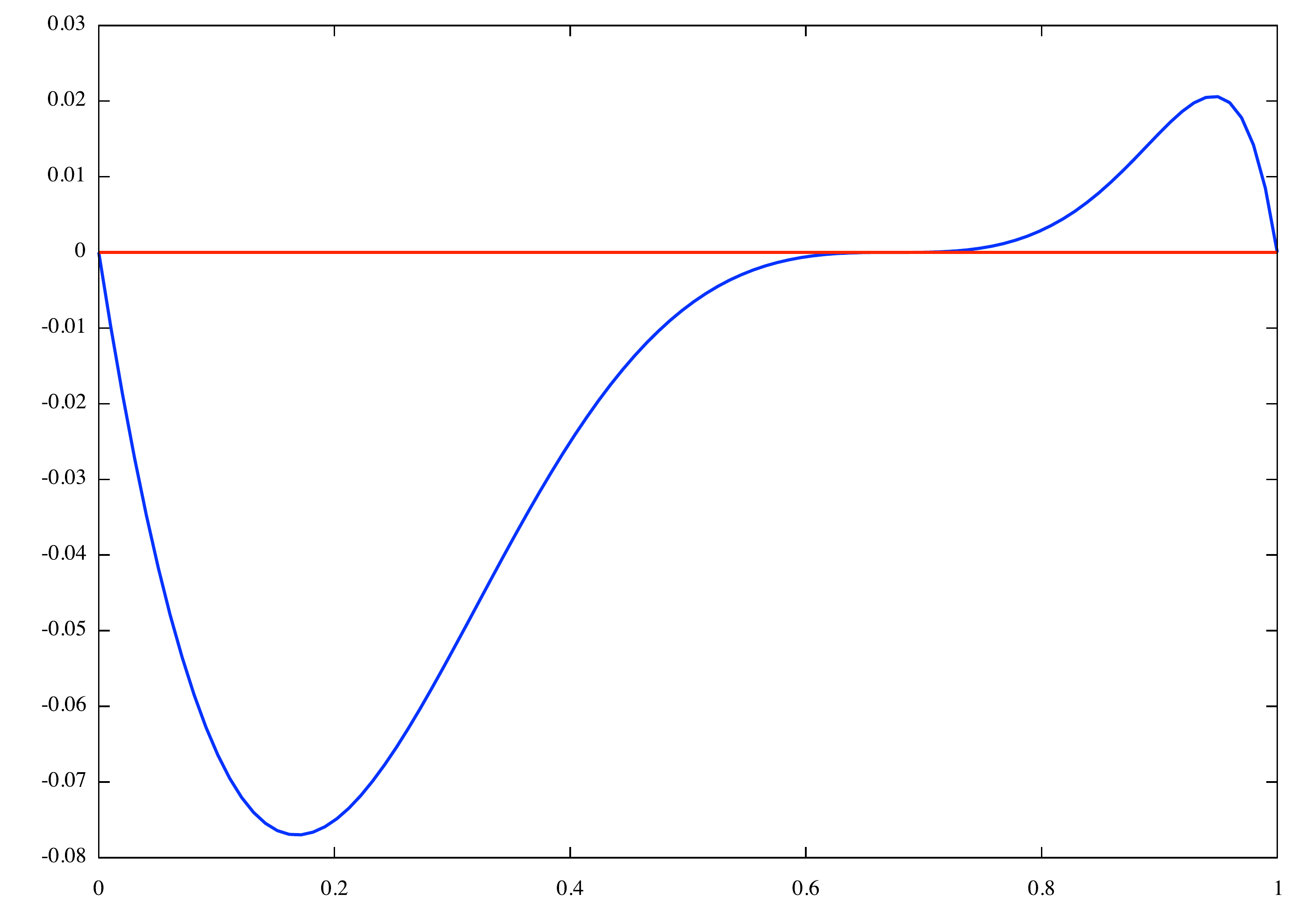}
\end{center}
\caption{
\label{figure:critical}
The function $f(x)-x$ for $x\in[0,1]$ in three
further cases:
(a) $p_1=0.5$, $p_2=0.25$, $p_4=0.25$;
(b) $p_2=0.783$, $p_{30}=0.217$, and
(c) $p_2=0.705$ and $p_{12}=0.295$. 
}
\end{figure}
For these distributions (the second and third are only approximate), 
we see points $q\in Q$ with $f'(q)=1$, and so the rescalings 
of Theorem \ref{thm:rescaled}(b), which are polynomial
rather than exponential, apply. 
In Figure \ref{figure:critical}(a) the relevant fixed points
are at $0$ and $1$, and in Figure 
\ref{figure:critical}(a), 
they appear as ``touchpoints" in the graph of $f(x)-x$ and so are unstable from one side only;
in these cases only one of the points $a$ and $-a$ in Theorem 
\ref{thm:rescaled}(b)
receives positive mass.
By contrast,
in Figure \ref{figure:critical}(c), the point of inflection gives a fixed point which is unstable on both sides. 

In passing, we mention briefly another interesting
aspect of some of the above examples, concerning
phase transitions. As we vary the offspring distribution,
we see points where, for example, the number of atoms
of the limiting distribution $W$ changes. Often, 
the transition can be \textit{continuous}: as the offspring distribution is varied,
an existing atom may split into several new atoms
(as may happen when a point of inflection occurs as 
in Figure \ref{figure:ternary}(d) or Figure \ref{figure:critical}(c)), or new atoms may appear whose
weight grows continously from $0$ (such as happens at the points $0$ and $1$ in Figure \ref{figure:ternary}(b)).
On the other hand, we can also see \textit{discontinuous}
transitions in cases such as Figure \ref{figure:critical}(b);
one can perturb the offspring distribution in an arbitrarily small way to remove the ``touchpoints" seen there, so that 
the atoms of $W$ inside $(0,1)$ disappear and all there mass jumps to the endpoints $0$ and $1$. Such ideas, expressed
only vaguely here, are studied in a closely related context in \cite{HolroydMartin}.

%Another feature which we observe briefly about Figure \ref{figure:critical} 
%is the nature of the phase transition. 
%
%
%
%
%The phase transitions observed around the critical points 
%in Figure \ref{figure:ternary}(b), (d) and Figure 
%\ref{figure:critical}(a) are all ``continuous"; 
%new fixed points emerge continuously from the positions
%of old fixed points. Here (approximately) we see a discontinuous
%fixed point; new roots appear suddenly far away from 
%the locations of pre-existing roots. At the critical point itself
%there are two half-stable fixed points. In both (a) and (b),
%the scaling limits around these fixed points
%are one-sided.
%(c) A variation of Figure \ref{figure:b2}(b): 
%now 
%This is again (approximately) critical; we
%see a fixed point which is unstable from both sides,
%but with slope 0. Although $W$ is constant, do we get a different sort of scaling limit from the one in the regular case,
%or in Figure \ref{figure:b2}(a)?

\subsection{Atoms at endpoints}\label{subsec:endpoints}
The limiting distribution $W$ in Theorem \ref{thm:unscaled} may have atoms at 0 and 1. We note the following simple criterion:

\begin{proposition}
\label{proposition:0atom}
Let $\mu$ be the mean of the offspring distribution.
\begin{itemize}
\item[(i)]If $p_1\mu<1$ then $\P(W=0)=\P(W=1)=0$. 
\item[(ii)]If $p_1\mu>1$ (including the case
$p_1>0$ and $\mu=\infty$) then $\P(W=0)>0$ and $\P(W=1)>0$. 
\end{itemize}
\end{proposition}
If $p_1\mu=1$, or if $p_1=0$ and $\mu=\infty$,
either case is possible. 
The proof of the result is straightforward. Since $f=R\circ R$ 
we have $f'(x)=R'(R(x))R'(x)$. Then since $R(0)=1$ and $R(1)=0$,
and since $R'=-G'$, we have $f'(0)=f'(1)=G'(0)G'(1)=\mu p_1$
(assuming $\mu<\infty$); and we know that a fixed point $q$ of $f$
is an atom of $W$ if $f'(q)>1$, and not if $f'(q)<1$. 

There is a rather direct interpretation of the condition
$p_1\mu>1$ in terms of the Galton-Watson tree and the play
of the game. Consider the set of paths in the tree,
starting at the root, with the following property:
every vertex along the path at an odd level has only one child. 
The union of these paths gives a subtree containing
the root. For a vertex at an even level
(such as the root), the expected number of 
grandchildren in the subtree is $p_1\mu$, 
since the vertex itself has an average of $\mu$
children, and each of those has precisely one child
with probability $p_1$. Considering
only even levels, this then gives a branching
process with mean offspring $p_1\mu$; 
if $p_1\mu>1$, then this branching process is supercritical
and survives for ever with positive probability.
In that case, by keeping the game within this tree,
the first player can ensure that
the second player never has any choice at all;
all the second player's moves are forced. 
For the game truncated at level $2n$, 
the first player can choose between 
all the nodes at level $2n$ which are within the
subtree; from this it can be shown that
$\P(W=0)$ is at least as big as the probability 
that this branching process survives. 

%\begin{figure}[htb]
%\begin{center}
%\includegraphics[width=0.49\textwidth]{{b2}.pdf}
%\includegraphics[width=0.49\textwidth]{{2and12}.pdf}
%\end{center}
%\caption{
%\label{figure:b2}
%(a) Regular binary tree, $p_2=1$. The points 0 and 1
%are stable, and there is a unique unstable fixed point.
%The distribution of $W$ is concentrated at a single point.
%(b) $p_2=0.6$, $p_{12}=0.4$. Now we see five fixed points (including 0 and 1), of which
%two are unstable. The limiting distribution of $W_{2n}$ 
%is concentrated on those two points. 
%}
%\end{figure}

\subsection{The case  \texorpdfstring{$f(x) \equiv x$}{f(x) equivalent to x}, and related open questions}
\label{subsec:comments_identity}
Suppose the offspring distribution is such 
that $f$ is the identity function. 
Then from (\ref{quantilerecursion}),
if we put independent values at the leaves
from any given distribution, then the value
at the root has that same distribution
(hence the statement in Theorem \ref{thm:unscaled}(a)). 
Perhaps surprisingly, this property is not restricted
to the trivial case where $p_1=1$. 

%The function $R$ is a bijection from $[0,1]$ to $[0,1]$,
%and $f=R\circ R$. A function that when composed with itself gives the identity is sometimes called an \textit{involution}.

Here are some families of examples
where $f=R\circ R=(1-G)\circ(1-G)$ is the identity
(i.e.\ $R$ is an \textit{involution}):
\begin{itemize}
\item[(a)]
Any geometric distribution. If $p_k=p(1-p)^{k-1}$
for $p\in(0,1)$, then $G(x)=\frac{p}{1-(1-p)x}$ and so 
$R(x)=\frac{1-x}{1-(1-p)x}$, and one can check 
$f(x)=x$. 
\item[(b)]
Let $G(x)=\left[1-(1-x)^{1/n}\right]^n$,
for $n=1,2,3,\dots$. Via a binomial expansion,
one can express $G$ has a power series expansion
with non-negative coefficients, and $G(1)=1$, 
so $G$ is indeed a probability generating function.
The coefficient of $x^k$ is non-zero for $k\geq n$. 
\item[(c)]
Let $G(x)=1-\left(1-x^n\right)^{1/n}$, for $n=1,2,3,\dots$.
Again $G$ has a power series expansion with non-negative
coefficients summing to 1. The coefficient of $x^k$ 
is non-zero when $k$ is a multiple of $n$. 
\end{itemize}

These are far from the only cases. For a general source of examples, consider function $S(x,y)$ from $[0,1]^2$ to $[0,1]$ which is
symmetric, increasing in each coordinate, and has $S(1,0)=S(0,1)=0$. If we define a function $R$ by setting $S(x,y)=0$ and writing $y=R(x)$, then $R$ is indeed an involution. Some such functions $R$ have power series expansions, and in some of those cases $G=1-R$ has all 
coefficients positive, as needed for a probability 
generating function. For example, $S(x,y)=y^2+y+x^2+x-2=0$ gives $R(x)=[\sqrt{9-4x-4x^2}-1]/2$, in which case 
one can obtain straightforwardly that $G=1-R$ is a generating function. 

We note several questions that it might be interesting
to understand further:
\begin{itemize}
\item[(1)]
Can one describe in some nice way the 
class of all distributions for which $f$ is the identity?
For the class of examples described in the previous paragraph,
can one describe nicely which functions $S(x,y)$ lead 
to functions $R$ which have power series expansions, and 
then which of those yield a generating function $G$?
\item[(2)]
Are the geometric distributions in example (a) above
the only such distributions with finite mean?
More generally, what types of tail decay are possible? 
For (a), the tail $\sum_{r=k}^\infty p_r$ of course decays exponentially in $k$, 
while for (b) and (c) it decays as $k^{-1/n}$.
\item[(3)]
Are there direct probabilistic arguments explaining the 
fact that $f$ becomes the identity in these cases, in terms of the underlying process on the tree? 

One case where it's possible to make such an argument is the $n=2$ case in (b) above. Here $p_k$
is the probability that the cluster containing the origin
has size $k$ for critical percolation on the binary tree (these coefficients are closely related to the Catalan numbers). 
Having made this identification, one can connect the 
minimax recursion on our random tree to an analogous
recursion in the model studied by Pemantle and Ward \cite{Pemantle},
of a binary tree in which each node independently is a max or a a min with probability $1/2$ each. 
\end{itemize}

We end this section with two further open questions 
about the form of $f$ in more general cases:
\begin{itemize}
\item[(4)] 
Can $f$ have an arbitrarily large number of fixed points in $[0,1]$?
\item[(5)]
Can $f$ have infinitely many fixed points in $[0,1]$ (without being equal to the identity)? Since $f$ is analytic on $(0,1)$,
this would require the set of fixed points to accumulate 
at $0$ and at $1$. 
\end{itemize}

\section{Endogeny}
\label{sec:endogeny}
Suppose we play the game on a tree where the depth $2n$ is large,
so that the boundary values are far from the root. 
To be confident of making a good first move, do we need to see 
a large part of the structure of the tree, and the boundary values? -- or can we play close to optimally by inspecting
just the structure of the first few levels of the tree?
This is a so-called \textit{endogeny} question \cite{AldBan}.
The answer to this question again depends 
on the offspring distribution and
the distribution of the boundary values. 

To formalise the question,
first define an operator on distributions 
corresponding to the recursion given by 
(\ref{eq:rde}). For a distribution $\mu$
on $[0,1]$, let $T(\mu)$ be the distribution 
of the left-hand side of (\ref{eq:rde})
when the random variables $W_{2n-n}^{(i,j)}$ 
on the right-hand side are i.i.d.\ with 
distribution $\mu$. Equivalently, 
rewriting (\ref{quantilerecursion}),
$T(\mu)[0,y]=f(\mu[0,y])$ for all $y$. 

We will be interested in fixed points of $T$. 
For example, for offspring distributions such that
$f$ is the identity, \textit{every} $\mu$ is a fixed point of $T$.
For more general offspring distributions, 
whenever $x$ is a fixed point of $f$, the Bernoulli distribution
which puts mass $x$ at $0$ and $1-x$ at $1$ is a fixed point of 
$\mu$; for a game with Bernoulli terminal values, there are only
two possible values of the outcome and we can interpret $0$ 
as a win and $1$ as a loss (from the perspective of the first player).

Suppose indeed that $\mu$ is a fixed point of the operator $T$. 
Consider a tree of depth $2n$ 
(given by the Galton-Watson tree truncated at level $2n$)
with the terminal values drawn independently from $\mu$.
Then the distribution of the value at the root is also $\mu$.
More generally, consider the structure of the first $k$ levels
of the tree;
the distribution of these first $k$ levels is the same
for any $n$ (such that $k\leq 2n$).

As a consequence of this \textit{consistency} 
over different values of $n$, we may let $n\to\infty$
and, applying Kolmogorov's extension theorem, obtain
a distribution of the entire infinite Galton-Watson tree
along with values attached to each node which
obey the minimax recursions (min at even levels,
max at odd levels). 

This gives a \textit{stationary recursive tree process}
in the language of \cite{AldBan}. The relevant stationarity
property is the following: condition on the structure of 
the first two levels of the tree, and write 
$v_1, \dots, v_r$ for the level-$2$ nodes. 
Conditional on the structure of the first two levels, 
the structure of the subtrees rooted at $v_1, \dots, v_r$,
along with the values associated to the nodes of those
subtrees, are given by $r$ i.i.d.\ copies of the original tree process.
(More precisely, we might describe the tree process as ``2-periodic"
rather than stationary, since even and odd levels differ; we
can recover a genuinely stationary process by considering only even levels.) 

For a more formal and more general set-up, see for 
example \cite{AldBan} or \cite{MachSturmSwart}. 

We have defined a joint distribution of the structure
of the tree and the values associated to each node of the tree. 
Now the recursive tree process is said to be \textit{endogenous} if the value associated to the root
is measurable with respect to the structure of the tree. 
Note that for the same offspring distribution, 
this endogeny property may hold for some 
fixed point distributions $\mu$ and not for others. 

Being measurable with respect to the structure of the
tree is equivalent to being approximable to any 
given degree of accuracy using the information only of
some finite portion of the tree. That is, 
for any random variable $X$ (in particular, the root value),
$X$ is measurable with respect to the structure of the tree if,
for any $x$ and any $\epsilon>0$,
there exists $k$ such that with probability $1-\epsilon$,
the conditional probability of the event $\{X\leq x\}$, 
given the structure of the first $k$ levels of the tree
is in $[0,\epsilon]\cup[1-\epsilon,1]$, where $X$ denotes
the value at the root.

For a more concrete interpretation, we can concentrate
only on the case of finite trees, truncated at some level 
$2n$. Then the property in the previous paragraph can be
reformulated to say that the value at the root can be 
approximated arbitrarily closely using information
from the structure of some appropriate number of levels 
at the top of the tree, \textit{uniformly} in the value of $n$. 

Note that endogeny does \textit{not} indicate that
the value at the root is insensitive to arbitrary changes
in the boundary conditions. In our case, that would
be trivially false. Rather, for a given distribution
of boundary conditions, endogeny expresses the property
that, if the boundary is far away, the root is typically not much affected by the difference between various 
realisations drawn from that distribution. In particular, 
endogeny may hold for some boundary distributions and
not for others, as is indeed the case for our model.

Consider in particular the Bernoulli (``win/loss") boundary
conditions described above. 

\begin{theorem}\label{thm:endogeny}
Let $x\in(0,1)$ be a fixed point of $f$,
and consider the stationary recursive tree process 
with Bernoulli($1-x$) marginals for the values at even levels.
The process is endogenous if and only if $f'(x)\leq 1$.
\end{theorem}

So, approximately speaking, the endogenous processes with Bernoulli
marginals correspond to the \textit{stable}
fixed points of the function $f$,
which are those fixed points which do \textit{not}
appear as atoms in the distribution of the limiting 
random variable $W$ in Theorem \ref{thm:unscaled}.
(An exception may occur when the derivative of $f$ at a fixed point is precisely 1; further, in the cases $x=0$ and $x=1$ the values are constant and the process is trivially endogenous.)

To prove Theorem \ref{thm:endogeny}, we use
a characterisation of endogeny in terms of 
uniqueness of bivariate distributions,
introduced by Aldous and Bandyopadhyay in \cite{AldBan} and proved in somewhat more generality by Mach, Sturm and Swart \cite{MachSturmSwart}.
See Section \ref{sec:endogeny_proof} for details.

For offspring distributions where $f$ is the identity, 
any distribution $\mu$ gives rise to a recursive tree process. 
In particular, we can take $\mu$ to be the uniform distribution
on $[0,1]$, as we did in previous sections.
We have the following corollary of Theorem \ref{thm:endogeny}:

\begin{corollary}\label{cor:endogeny}
Suppose $f$ is the identity. Then for any $\mu$, the recursive tree process with marginals $\mu$ for the values at even levels
is endogenous.
\end{corollary}

\section{Proofs: convergence and scaling limits}
\label{sec:distribution_proofs}
First, we show how (\ref{quantilerecursion})
follows from 
the recursive distributional equation (\ref{eq:rde}). As at (\ref{eq:rde}),
let $M$ and $M_1, M_2, M_3$ be i.i.d.\ draws 
from the offspring distribution, and $W_{2n-2}^{(i,j)}$ 
i.i.d.\ copies of the random variable $W_{2n-2}$, 
independent of $M$ and $\{M_i, i\geq 1\}$.

Note that for any given $i$,  
\begin{align}
\nonumber
\P\left(
\max_{1\leq j\leq M_i} W_{2n-2}^{(i,j)}>x
\right)
&=1-\P\left(W_{2n-2}^{(i,j)}\leq x
\text{ for } j=1,\dots,M_i\right)\\
\nonumber
&=1-\sum_m p_m\P\left(W_{2n-2}\leq x\right)^m\\
&=R\big(\P(W_{2n-2}\leq x)\big).
\label{onestep}
\end{align}

So from (\ref{eq:rde}) we have
\begin{align}
\P\left(W_{2n}\leq x\right)
&=\P\left(
\min_{1\leq i\leq M}
\max_{1\leq j\leq M_i}
W^{(i,j)}_{2n-2}\leq x
\right)
\nonumber
\\
&=1-\P\left(
\max_{1\leq j\leq M_i}
W^{(i,j)}_{2n-2}
>x \text{ for all }j=1,\dots, M_i
\right)
\nonumber
\\
&=1-\sum_m p_m\left[R\left(\P\left(W_{2n-2}\leq x\right)\right)\right]^m
\nonumber
\\
&=R\big(R\big(\P(W_{2n-2}\leq x)\big)\big)
\nonumber
\\
\nonumber
&=f\big(\P(W_{2n-2}\leq x)\big),
%\label{quantilerecursion}
\end{align}
giving (\ref{quantilerecursion}) as desired. 

\begin{proof}[Proof of Theorem \ref{thm:unscaled}]
From the previous line and the monotonicity of $f$ we see that $\lim_{n\to\infty} \P(W_{2n} \leq x)=\lim_{n\to\infty} f^n(x)$ exists for all $x$, and therefore $W_{2n}$ indeed converges in distribution as $n \to \infty$, to a limit $W$ with the distribution function $F_W(x)=\lim_{n\to\infty} f^n(x)$.  
%By continuity of $f$, $F_W(x)$ is a fixed point of $f$.  

Part (a) is immediate from (\ref{quantilerecursion}).
For part (b), note that since $f$ is analytic in $(0,1)$ and $f$ is not the identity function, the set of fixed points of $f$ 
cannot have an accumulation point in $(0,1)$. 
Therefore, this set of fixed points of $f$ defines a partition of the interval $(0,1)$ into a disjoint union of open intervals plus the set of fixed points, each of which is an endpoint of exactly two intervals from the partition. 
Since $f$ is monotone and continuous,
$F_W(x) = \lim_{n \to \infty} f^n (x)$ is constant on those intervals; therefore $W$ can have atoms only at fixed points of $f$. 

Suppose $q\in(0,1)$ is such a fixed point. Then 
\begin{align}
\nonumber
\P(W=x)
&=
\lim_{\epsilon\to 0}
\P(q-\epsilon<W\leq x+\epsilon)\\
&=\lim_{\epsilon\to 0}\lim_{n\to\infty}f^n(q+\epsilon) - \lim_{\epsilon\to 0} \lim_{n\to\infty}f^n(q-\epsilon).
\label{eq:PW}
\end{align}
Since $f$ is monotone and continuous, the quantity above
is equal to 0 precisely if and only if the fixed point $q$ is stable. 
Hence indeed $W$ has an atom at $q$ precisely if $q$ 
is unstable from at least one side.
As commented immediately after Theorem \ref{thm:unscaled},
the right-hand side of (\ref{eq:PW}) is equal to
$q_+-q_-$, as required.

The cases where $q=0$ or $q=1$ follow in a similar way.
\end{proof}

The rest of this section is devoted to the proof of Theorem 
\ref{prof:prop_main}.

\subsection{Proof of Theorem 
\ref{prof:prop_main}(\ref{thm:main_greater_t_1}): \texorpdfstring{$ 1 < f'(q) < \infty$}{1 < f'(q) < infinity}} 
\label{subsec:scaling_proof}
Firstly we assume that $q$ is the unique fixed point of $f$ in $(0,1)$ and that it is unstable from both sides. In the second part of the proof we show how Lemma \ref{lm:khan_extension} below implies the general case.

Suppose $\xi = f'(q) > 1$. An example of this case is in Figure
\ref{figure:ternary}(a), where $p_1=0.45$ and $p_3=0.55$.

We will prove the following result:
\begin{lemma}
Consider the recursion (\ref{quantilerecursion}) and assume that $q$ is the unique fixed point of $f$ in $(0,1)$ and that it is unstable from both sides. Set $\xi = f'(q)$. If $\xi > 1$, then
\begin{align*}
\xi^n(W_{2n} - q) \xrightarrow[]{d} V \textrm{ as } n \rightarrow \infty,
\end{align*}
where the distribution function $F_V$ of $V$ is continuous and satisfies
\begin{align*}
F_V(x) = f(F_V(x/\xi)), \quad x \in \R.
\end{align*}
\label{lm:khan_extension}
\end{lemma}
Lemma \ref{lm:khan_extension} extends the result of Ali Khan, Devroye and Neininger \cite{ADN} to the case of random trees. Note that Lemma \ref{lm:khan_extension} corresponds directly to the part (\ref{thm:main_greater_t_1}) of Theorem \ref{prof:prop_main} for $f$ having a unique fixed point in $(0,1)$ which is unstable, as then $q_- = 0$ and $q_+ = 1$.

\begin{proof}[Proof of Lemma \ref{lm:khan_extension}]
We follow the lines of \cite{ADN} but in our case the analysis is slightly more complicated because of the more general form that $f$ can admit. We first prove that there exists a pointwise limit of distribution functions of $\xi^n(W_{2n} - q)$, which is not identical to $0$ or $1$, and then show that it is continuous, which completes the proof. Define  
\begin{align*}
g_n(x) = \P \left(\xi^n(W_{2n} - q) \leq x \right), \quad x \in \R
\end{align*}
Therefore, for each $x$ for sufficiently large $n$ (such that $0 \leq q + \frac{x}{\xi^n} \leq 1$),
\begin{align*}
g_n(x) = \P \left(W_{2n} \leq q + \frac{x}{\xi^n} \right)= f^n \left(q + \frac{x}{\xi^n} \right).
\end{align*}
Note that $g_n(0) = q $ for all $n$. We need some local uniform bound for $g_n$ around $x = 0$. This will be supplied by the following lemma:
\begin{lemma}
Under the assumptions of Lemma \ref{lm:khan_extension}, let $k$ be the smallest number larger than $1$ such that $f^{(k)}(q) \neq 0$. Denote $h_1(x) = q + x$ and $h_2(x) = q + x + cx^k$ for $x \in \R$. Then there exist $c$ and an $\varepsilon > 0$ such that for all $n$ and $|x| \leq \varepsilon$,  either $h_1(x) \leq g_n(x) \leq h_2(x)$ or $h_2 (x) \leq g_n(x) \leq h_1(x)$.
\label{lm:bound_on_g}
\end{lemma}
Note that such a number $k$ exists since we assumed that $f$ is not the identity function and $f$ is analytic at $q$.
\begin{proof}[Proof of Lemma \ref{lm:bound_on_g}]
Take any $c$ such that $c$ has the same sign as $f^{(k)}(q)$  and $|c| > \left|\frac{f^{(k)}(q)}{k! \xi (\xi^{k-1}-1)} \right|$. From analyticity of $f$, $f^{(k)}(x)$ does not change the sign on some neighbourhood of $q$. 

For simplicity assume that $k$ is even and $f^{(k)}(q) > 0$. We would generally need to consider four cases depending on the parity of $k$ and the sign of $f^{(k)}(q)$. For the other three cases the steps of the proof of the lemma are identical modulo the change of sign in the inequalities.

The proof is by induction on $n$ and makes use of Taylor's formula up to order $k$.  For $n=0$ the assertion is true, as we have
\begin{align*}
h_1(x) = q + x = g_0(x) \leq h_2(x).
\end{align*}
Note that the above holds for all $\varepsilon$, thus we will chose $\varepsilon$ later. Assume now that 
\begin{align*}
h_1(x) \leq g_{n-1}(x) \leq h_2(x)
\end{align*}
for some $n-1 \geq 0$, $|x| \leq \varepsilon$ and $\varepsilon > 0$. Since $\left| \frac{x}{\xi} \right| \leq \varepsilon$, as $|x| \leq \varepsilon$ and $\xi > 1$, and $f$ is increasing, we have
\begin{align*}
g_n(x) = f^n \left(q + \frac{x}{\xi^n} \right) = f \left(f^{n-1} \left(q + \frac{x/\xi}{\xi^{n-1}} \right)\right)= f \left(g_{n-1}\left(x/\xi \right)\right)\geq f \left(h_1 \left(x/\xi \right)\right), 
\end{align*}
and analogously
\begin{align*}
g_n(x) \leq f \left(h_2 \left(x/\xi \right)\right). 
\end{align*}
The induction proof will be completed if we can show that for some $\varepsilon > 0$, for $|x| \leq \varepsilon$,
\begin{align}
h_1(x) \leq f \left(h_1 \left(x/\xi \right)\right), \quad f \left(h_2 \left(x/\xi \right)\right)\leq h_2(x). 
\label{eq:h_less_f_of_h}
\end{align}
Taking the Taylor expansion of $f$ around $q$ at points $q + x/\xi$ and $q + x/\xi + c \left(x/\xi\right)^k$, we obtain
\begin{align*}
f(h_1(x/\xi)) = {} & q + x +\frac{1}{k!} \frac{1}{\xi^k} \left( f^{(k)}(q) \right)x^k + o(x^{k}),\\
f(h_2(x/\xi)) = {} & q + x +\frac{1}{k!} \frac{1}{\xi^k} \left( f^{(k)}(q) + \xi k! c\right)x^k + o(x^{k}).
\end{align*}
Since
\begin{align*}
0 < \frac{1}{k!} \frac{1}{\xi^k} \left( f^{(k)}(q) + \xi k! c\right) = \frac{f^{(k)}(q)}{k!\xi(\xi^{k-1}-1)} \frac{\xi^{k-1}-1}{\xi^{k-1}} + c\frac{1}{\xi^{k-1}} < c
\end{align*}
and by assumption $f^{(k)}(q) > 0$, we are therefore able to pick $\varepsilon > 0$ such that (\ref{eq:h_less_f_of_h}) holds for $|x| \leq \varepsilon$. This ends the proof of Lemma \ref{lm:bound_on_g}. 
\end{proof}
Note that if we show that for some $g$, $g_n(x) \rightarrow g(x)$ for all $x$, then the above lemma will imply that $g(x)$ is continuous and differentiable at $x=0$ with $g'(0)=1$. We now claim that for each $x$, $(g_n(x))$ is a monotone sequence for $n > n_x$. This is implied by the following lemma:

\begin{lemma}
Under the assumptions of Lemma \ref{lm:khan_extension}, for each $M$ there exists $n_M$ such that for $|x| \leq M$, $(g_n(x))$ is a monotone sequence for $n \geq n_M$.
\label{lm:monotonicity_g_n}
\end{lemma}
\begin{proof}[Proof of Lemma \ref{lm:monotonicity_g_n}]
As in the proof of Lemma \ref{lm:bound_on_g}, we consider the case where $k$ is even and $f^{(k)}(q) > 0$ -- the other cases are identical.
Using Taylor expansion up to order $k$, there exists $\varepsilon > 0$ such that for $|y| \leq \varepsilon $, 
\begin{align}
f\left(q + y\right) \geq f(q) + f'(q)y
\label{eq:taylor_bound}
\end{align}
(recall that $f^{(i)}(q) = 0$ for $1 < i < k$). Now let $n_M = \lceil \log_{\xi}\left(\frac{M}{\varepsilon}\right)\rceil$ and note that for any $|x| \leq M$ and $n \geq n_M$, $|x/\xi^n| < \varepsilon$ and therefore by (\ref{eq:taylor_bound}),
\begin{align*}
f \left(q+\frac{x}{\xi^n} \right) \geq f(q) + f'(q)\frac{x}{\xi^{n}} = q + \frac{x}{\xi^{n-1}}.
\end{align*}
Finally, since $f^{n-1}$ is monotone increasing,
\begin{align*}
g_n(x) = f^n \left(q + \frac{x}{\xi^n} \right) = f^{n-1}\left(f\left(q+\frac{x}{\xi^n}\right)\right) \geq f^{n-1}\left(q+\frac{x}{\xi^{n-1}}\right) = g_{n-1}(x).
\end{align*}
This proves the claim. 
\end{proof}

Since $g_n(x) \leq 1$, by Lemma \ref{lm:monotonicity_g_n} $g_n(x)$ converges for all $x$ -- we denote the limit by $g(x)$. The continuity of $f$ and the fact that $g_n(x) = f(g_{n-1}(x/\xi))$ imply that
\begin{align}
g(x)=f(g(x/\xi)).
\label{eq:identity}
\end{align} 
Therefore, from the continuity of $f$ and the monotonicity of $g$, $\lim_{x \to -\infty} g(x)$ and $\lim_{x \to \infty}g(x)$ are fixed points of $f$. Using the fact that $\{0,q,1\}$ are the only fixed points of $f$, $g$ is non-decreasing, $g(0) = q$ and $g'(0) = 1$, we deduce that $\lim_{x \to -\infty} g(x) = 0$ and $\lim_{x \to \infty}g(x) = 1$. When we show that $g$ is continuous at all $x$, it will then imply that $F_V = g$. 

%\paragraph{Continuity of $g(x)$:} The idea of the proof is as follows. 
We apply the following strategy to show that $g$ is continuous: we showed that $g(x)$ is continuous at $0$ and now we show separately that it is continuous on some $(-\varepsilon, 0)$ and on some $(0, \varepsilon)$. 
The identity (\ref{eq:identity}), together with the continuity
of $f$, then implies that $g$ is continuous on all of $\R$
(since $\xi>1$). 

We still work under the assumption that $f^{(k)}(q) > 0$, where $k\geq 2$ is such that 
$f^{(r)}(q)=0$ for $1<r<k$ and $f^{(k)}(q)\ne 0$, and that $k$ is even (if $k$ is odd then reasoning in the two cases below should be swapped). Note that to prove that $g$ is continuous on some interval $I$, it is sufficient to show that
\begin{align}
\sup_{y \in I} \sup_{n \geq 0} g_n'(y) < \infty.
\label{eq:sup_derivative}
\end{align}
By the chain rule we obtain
\begin{align}
g_{n}'(x) = \frac{1}{\xi^n} \prod_{i=0}^{n-1} f'\left(f^i \left(q + \frac{x}{\xi^n}\right)\right).
\label{eq:chain_rule}
\end{align}

%First we consider $y>0$ in (\ref{eq:sup_derivative}).
We consider first $g_n'(y)$ for $y < 0$. Since $f^{(k)}(q) > 0$, there exists an $\varepsilon>0$ such that $f'(q+y) < \xi$ for $y \in (-\varepsilon,0)$. Since $f(q) = q$ and $\xi > 1$, this implies that for $i < n$
\begin{align*}
q > f^i \left(q + \frac{y}{\xi^n}\right) > q + \xi^i \frac{y}{\xi^n} > q - \varepsilon,
\end{align*}
hence $f'\left(f^i \left(q + \frac{y}{\xi^n}\right)\right) < \xi$. By (\ref{eq:chain_rule}) we conclude that $g_n'(y) < 1$ for all $n$ and $y \in (-\varepsilon, 0)$. This implies that (\ref{eq:sup_derivative}) holds with $I = (-\varepsilon,0)$, hence $g$ is continuous on $(-\varepsilon, 0)$.

Now we turn to the case of $y>0$.
The function $f$ is non-decreasing, and 
$\xi > 1$; hence for all $0 < i < n$ and all $0 \leq y \leq x$,
\begin{align}
q  \leq f^i\left(q + \frac{y}{\xi^n}\right) \leq f^i\left(q + \frac{x}{\xi^n}\right).
\label{eq:case_y_geq_0_1}
\end{align}
Note also that
\begin{align}
f^i\left(q + \frac{x}{\xi^n}\right) \leq f^i\left(q + \frac{x}{\xi^i}\right) = g_i(x).
\label{eq:case_y_geq_0_2}
\end{align}
Now by the assumption $f^{(k)} > 0$ there exists an $\varepsilon > 0$ such that $f'(q+x)$ is strictly increasing for $x \in (0, \varepsilon)$. By the continuity of $g$ at $0$, there exists $\gamma > 0$ such that $g(x) < q + \varepsilon$ for $x \in (0, \gamma)$. By Lemma \ref{lm:monotonicity_g_n}, there exists $n_\gamma$ such that for $0 < x \leq \gamma$, $(g_i(x))$ is a monotone sequence (an increasing one, since we assume $f^{(k)}(q) > 0$) for $i > n_\gamma$.
Therefore, for $i \geq n_\gamma$, and for $0 < x < \gamma$,
\begin{align}
g_i(x) \leq g(x) < q + \varepsilon.
\label{eq:case_y_geq_0_3}
\end{align}
On the other hand, for $i < n_\gamma \wedge n$, since $f(q+x) > q + x$ for $x \in (0, 1-q)$,
\begin{align}
f^i\left(q + \frac{x}{\xi^n}\right) \leq f^{n_\gamma}\left(q + \frac{x}{\xi^n}\right) \leq f^{n_\gamma} \left(q + x\right).
\label{eq:case_y_geq_0_4}
\end{align}
$f^{n_\gamma}$ is continuous and non-decreasing, hence we may pick $\tilde \gamma > 0$ such that for $0 < x < \tilde \gamma$,
\begin{align}
f^{n_\gamma} (q + x)< q + \varepsilon.
\label{eq:case_y_geq_0_5}
\end{align}
Finally, combining (\ref{eq:case_y_geq_0_1}) -- (\ref{eq:case_y_geq_0_5}), we obtain that for all $0 < i < n$ and for all $0 \leq y \leq x \leq \gamma \wedge \tilde \gamma$,
\begin{align}
q  \leq f^i\left(q + \frac{y}{\xi^n}\right) \leq f^i\left(q + \frac{x}{\xi^n}\right) \leq q + \varepsilon.
\label{eq:case_y_geq_0_final_bound}
\end{align}
Recall that $\varepsilon$ was chosen to be such that $f'$ is strictly increasing on $(q, q+\varepsilon)$. Combining this with (\ref{eq:chain_rule}) and (\ref{eq:case_y_geq_0_final_bound}) we obtain that 
\begin{align}
g_n'(y) \leq g_n'(x).
\label{eq:g_n_prim_selfbound}
\end{align}
We are now going to use (\ref{eq:g_n_prim_selfbound}) to show that (\ref{eq:sup_derivative}) holds for all for $I = (0, \varepsilon)$ with $ \varepsilon = \frac{1}{2} (\gamma \wedge \tilde \gamma)$. For each $z \in ( \varepsilon, 2\varepsilon)$ and all $n \geq 0$, by (\ref{eq:g_n_prim_selfbound}) we have
\begin{align*}
\sup_{y \in (0,\varepsilon)} g_{n}'(y) \leq g_n'(z).
\end{align*}
Therefore,
\begin{align*}
 \sup_{n \geq 0}\sup_{y \in (0,\varepsilon)} g_n'(y) {} & \leq \sup_{n \geq 0} \frac{1}{\varepsilon} \int_{\varepsilon}^{2\varepsilon} g_n'(z)\textrm{d}z \\
{} & \leq \sup_{n \geq 0}\frac{1}{\varepsilon} (g_n(2\varepsilon) - g_n(\varepsilon) \\
{} & \leq  \frac{1}{\varepsilon},
\end{align*}
where the last inequality follows since each $g_n$ is a distribution function. This proves that $g(y)$ is continuous on $(0, \varepsilon)$. This completes the proof of Lemma \ref{lm:khan_extension}.

\end{proof}

\subsubsection*{Multiple atoms}
In the previous section we found the correct order of fluctuations when $f$ had a single fixed point in the interval $(0,1)$. When $f$ has more than one fixed point in the interval $(0,1)$, we cannot simply consider the quantity $W_{2n} - q$, since the limiting distribution has multiple atoms, but it turns out that we can condition on $W_{2n}$ being close enough to one of the atoms and straightforwardly apply Lemma \ref{lm:khan_extension}. Note that the set of fixed points of $f$ cannot have an accumulation point in the interval $(0,1)$. To see this, recall that $f$ is a composition of functions analytic in $(0,1)$. Therefore, $f(x) - x$ is also analytic in $(0,1)$ and we justify the claim using the fact that zeros of an analytic function not identical to $0$ cannot have any accumulation points in the domain in which the function is analytic. 

The case of multiple atoms of $V$ is summarized in the following lemma:
\begin{lemma}
Consider the recursion (\ref{quantilerecursion}) and assume that $q_{-},q,q_{+}$ are fixed points of $f$ satisfying the following conditions:
\begin{itemize}
\item $q \in (0,1)$ is unstable and $f'(q) > 1$,
\item $q_{-} < q < q_{+}$,
\item $q$ is the only unstable from at least one side point of $f$ in the interval $(q_{-}, q_{+})$.
\end{itemize}
Then, 
\begin{align*}
\mathcal{L} \left(\xi^n(W_{2n}-q)  \ | \ W_{2n} \in [q_-, q_+] \right) \xrightarrow[]{} \mathcal{L}(V),
\end{align*}
where $\xi = f'(q)$ and $V$ is a random variable with a continuous distribution function.
\label{lm:multiple_atoms}
\end{lemma} 
Note first that since $f'(q) > 1$, the definitions of $q_{-}$ and $q_{+}$ coincide with those given in (\ref{qminusdef}).

\begin{proof}[Proof of Lemma \ref{lm:multiple_atoms}]
Fix $x \in [0,1]$. Then 
\begin{align}
\begin{split}
\P\left(\frac{W_{2n}-q_{-}}{q_{+}-q_{-}} \leq x \Bigg|  W_{2n} \in [q_{-},q_{+}]\right) = {} & \frac{\P(q_{-} \leq W_{2n} \leq x(q_{+}-q_{-})+q_{-})}{\P(q_{-} \leq W_{2n} \leq q_{+})} \\
= {} & \frac{f(\P(W_{2n-2} \leq x (q_{+}-q_{-}) + q_{-}) - q_{-}}{q_{+}-q_{-}} \\
= {} & \frac{f\left(\P\left(\frac{W_{2n-2} - q_{-}}{q_{+}-q_{-}} \leq x\right)\right) - q_{-}}{q_{+}-q_{-}} \\ 
= {} & \tilde f\left(\P\left(\frac{W_{2n-2} - q_{-}}{q_{+}-q_{-}} \leq x \Bigg|  W_{2n-2} \in [q_{-},q_{+}]\right)\right),
\end{split}
\label{eq:recurence_for_combination}
\end{align}
where
\begin{align*}
\tilde f (x) = \frac{f(x(q_{+}-q_{-}) + q_{-}) - q_{-}}{q_{+}-q_{-}}.
\end{align*}
Furthermore, $\tilde f(x)$ is a continuous bijective mapping from $[0,1]$ to $[0,1]$ with a single fixed point $\tilde q = \frac{q-q_{-}}{q_{+}-q_{-}}$ in $(0,1)$ and
\begin{align}
\begin{split}
\tilde f'(\tilde q) = {} & f'(q) = \xi,\\
\tilde f ^{(k)} (\tilde q) = {} & f^{(k)}(q)(q_{+}-q_{-})^{k-1}.
\end{split}
\label{eq:assumptions_multiple_atoms}
\end{align}
Consider a sequence of random variables $(\tilde W_{2n})_{n=0}^\infty$ such that
\begin{align*}
\tilde W_{2n} \stackrel{d}{=} \left(\frac{W_{2n}-q_{-}}{q_{+}-q_{-}} \ \Bigg| \ W_{2n} \in [q_{-},q_{+}]\right).
\end{align*}
We check that $\tilde W_0 \sim U(0,1)$, $\P(\tilde W_{2n} \leq x) = \tilde f(\P (\tilde W_{2n-2} \leq x)$. Combining this with (\ref{eq:assumptions_multiple_atoms}), we may apply Lemma \ref{lm:khan_extension} to $(\tilde W_{2n})_{n=0}^\infty$ to conclude that
\begin{align*}
\mathcal{L} \left( \xi^n (W_{2n}-q) \ | \ W_{2n} \in [q_{-},q_{+}] \right) \xrightarrow[]{} \mathcal{L} ((q_{+}-q_{-})V),
\end{align*}
where $W$ is a random variable with a continuous distribution function $\tilde g$ that satisfies $\tilde g(x) = \tilde f(\tilde g(x))$. Finally, we note that the distribution function of $(q_{+}-q_{-})V$ is also continuous, which ends the proof.
\end{proof}

\subsubsection*{Boundary fixed points}
To finish the proof of part (\ref{thm:main_greater_t_1}) of Theorem \ref{prof:prop_main} we need to consider the case when one of $q_-, q_+$ is equal to $q$. This may happen either if $q$ is at the boundary (i.e. $q \in \{0,1\}$) or when $q \in (0,1)$, but $q$ is stable from one side. These cases can be treated simultaneously by repeating the reasoning from the proofs of Lemma \ref{lm:khan_extension} and Lemma \ref{lm:multiple_atoms}. Note that the limiting distribution $V$ is now concentrated on either the positive or negative half-line.

\subsection{Proof of Theorem 
\ref{prof:prop_main}(\ref{thm:main_eq_1}): \texorpdfstring{$f'(q) = 1$}{f'(q) = 1}}
We have already described the fluctuations of $W_{2n}$ when we know that it converges to some fixed point $q$ of $f$ with $f'(q) \in (1, \infty)$. If the point was unstable from both sides, we obtained a two-sided continuous limiting distribution.

If $q$ is a fixed point of $f$ such that $f'(q)=1$, it may be unstable, stable or unstable from one side and stable from the other. In this case it is more convenient to consider each side of $q$ separately. For simplicity, we state and prove a lemma for the case where $q$ is unstable from the right and then comment on the general case.

Note that the set of fixed points doesn't have an accumulation point in $(0,1)$, but it is not known whether this behaviour may be exhibited at the boundary, hence the additional assumption in the lemma.
\begin{lemma}
Consider the recursion (\ref{quantilerecursion}) and assume that $q$ is a fixed point of $f$, that it is unstable from the right and let $q_+ = \inf\{x: x>q, x=f(x)$. Suppose that $f'(q) = 1$ and $k$ is such that $f^{(r)}(q)=0$ for $1<r<k$ and $f^{(k)}(q)\ne 0$. If $q_+ \neq q$, then
%\marginpar{if $q=0$, can it happen that all derivatives are 0? or that don't exist? connected with the question about infinite number of fixed points} 
\begin{align*}
\mathcal{L} \left(n^{\frac{1}{k-1}}(W_{2n}-q)   \ | \ W_{2n} \in [q, q_+] \right)\stackrel{}{\to} \delta_a,
\end{align*}
where $a = \left( \frac{k(k-2)!}{f^{(k)}(q)} \right)^{\frac{1}{k-1}}$.
\label{lm:f_1_onesided}
\end{lemma}
Figure \ref{figure:critical}(a) gives
an example where 
Lemma \ref{lm:f_1_onesided} applies.

%Note that comparing to case (\ref{thm:main_greater_t_1}) where $f'(q) %> 1$, the rate of convergence has decreased and instead of a %continuous limiting distribution we obtained a deterministic limit.

\begin{proof}[Proof of Lemma \ref{lm:f_1_onesided}]
We are going to show that the distribution function of $n^{\frac{1}{k-1}}\frac{W_{2n}-q}{q_+-q}$ conditioned on the event $W_{2n} \in [q,q_+]$ converges to some limit as $n$ tends to infinity. Define
\begin{align}
g_n(x) := \P\left(n^{\frac{1}{k-1}}\frac{W_{2n}-q}{q_+-q} \leq x\ \bigg| \ W_{2n} \in [q,q_{+}]\right).
\label{eq:gn_f_1}
\end{align}
By calculations similar to those (\ref{eq:recurence_for_combination}) in the proof of Lemma \ref{lm:multiple_atoms}, we obtain that for $x \in [0, 1]$,
\begin{align*}
\P\left(\frac{W_{2n}-q}{q_+-q} \leq x \ \bigg| \ W_{2n} \in [q,q_{+}]\right) = \tilde f\left(\P\left(\frac{W_{2n}-q}{q_+-q} \leq x \ \bigg| \ W_{2n-2} \in [q,q_{+}]\right)\right),
\end{align*}
where
\begin{align*}
\tilde f (x) = \frac{f(x(q_+-q)+q) - q}{q_{+}-q}.
\end{align*}
Note that
\begin{align*}
\tilde f(0) = {} & 0,\\
\tilde f(1) = {} & 1, \\
\tilde f'(0) = {} & f'(q) = 1, \\
\tilde f^{(i)}(0) = {} & (q_+-q)^{i-1} f^{(i)}(q),
\end{align*}
and $\tilde f$ has no fixed points in $(0,1)$. Since for every $x > 0$, for sufficiently large $n$, $\frac{x}{n^{\frac{1}{k-1}}} \in [0,1]$, for such $n$ we have
\begin{align}
\begin{split}
g_n(x) = {} & \P\left(\frac{W_{2n}-q}{q_+-q} \leq  xn^{-\frac{1}{k-1}} \ \bigg| \ W_{2n} \in [q,q_{+}] \right)  \\
= {} & \tilde f^n \left( \P\left(\frac{W_{0}-q}{q_+-q} \leq xn^{-\frac{1}{k-1}}  \ \bigg| \ W_0 \in [q, q_+] \right)\right) \\
= {} & \tilde f^n \left(xn^{-\frac{1}{k-1}} \right).
\end{split}
\label{eq:gn_tilde_fn}
\end{align}
The proof consists of two parts:
\begin{enumerate}[I]
\item We show that for each $x<a$, for sufficiently large $n$, $(g_n(x))$ forms a decreasing sequence, and for each $x>a$, for sufficiently large $n$, $(g_n(x))$ forms an increasing one,
\item we show that for $x < a$, $g_n(x) \rightarrow 0$, and for $x > a$, $g_n(x) \rightarrow 1$.
\end{enumerate}

\paragraph{Part I}
Fix $x \neq 0$. Using Taylor's expansion, we may expand $\tilde f(x)$ as follows:
\begin{align*}
\tilde f(x) = \tilde f(0) + x + \frac{\tilde f^{(k)}(0)}{k!}x^k + r_k(x)x^k,
\end{align*}
where $\lim_{x \to 0} r_k(x) = 0$. Therefore, 
\begin{align}
\tilde f \left(\frac{x}{n^{\frac{1}{k-1}}} \right) < {} & \frac{x}{(n-1)^{\frac{1}{k-1}}}
\label{eq:f_eq_1_monotonicity}
\end{align}
is equivalent to
\begin{align*}
\frac{x}{n^{\frac{1}{k-1}}} + \frac{\tilde f^{(k)}(0)}{k!} \left(\frac{x}{n^{\frac{1}{k-1}}}\right)^k + r_k\left(\frac{x}{n^{\frac{1}{k-1}}}\right)\left(\frac{x}{n^{\frac{1}{k-1}}}\right)^k < {} & \frac{x}{(n-1)^{\frac{1}{k-1}}},
\end{align*}
and to
\begin{align*}
\frac{\tilde f^{(k)}(0)}{k!}x^{k-1} + r_k\left(\frac{x}{n^{\frac{1}{k-1}}}\right) x^{k-1} < {} & n \left(\left(\frac{n}{n-1}\right)^{\frac{1}{k-1}}-1\right).
\end{align*}
Letting $n \to \infty$, the right-hand side of the last formula converges to $\frac{1}{k-1}$, whereas the left-hand one converges to $\frac{\tilde f^{(k)}(0)}{k!}x^{k-1}$. Therefore, the last inequality is satisfied for large $n$ if
\begin{align*}
x < \left( \frac{k(k-2)!}{f^{(k)}(q)} \right)^{\frac{1}{k-1}} \frac{1}{q_+-q},
\end{align*}
and similarly
\begin{align}
\tilde f \left(\frac{x}{n^{\frac{1}{k-1}}} \right) > \frac{x}{(n-1)^{\frac{1}{k-1}}}
\label{eq:f_eq_1_monotonicity2}
\end{align}
for large $n$ if
\begin{align*}
x > \left( \frac{k(k-2)!}{f^{(k)}(q)} \right)^{\frac{1}{k-1}} \frac{1}{q_+-q}.
\end{align*}
This yields the claim, as $\tilde f^{n-1}$ is a strictly increasing function, hence recalling (\ref{eq:gn_tilde_fn}), the inequality (\ref{eq:f_eq_1_monotonicity}) is equivalent to 
\begin{align*}
g_n(x) = \tilde f^n \left(\frac{x}{n^{\frac{1}{k-1}}} \right) < {} & \tilde f^{n-1} \left (\frac{x}{(n-1)^{\frac{1}{k-1}}} \right) = g_{n-1}(x),
\end{align*}
and the inequality (\ref{eq:f_eq_1_monotonicity2}) is equivalent to
\begin{align*}
g_n(x) = \tilde f^n \left( \frac{x}{n^{\frac{1}{k-1}}} \right) > {} & \tilde f^{n-1} \left ( \frac{x}{(n-1)^{\frac{1}{k-1}}} \right) = g_{n-1}(x).
\end{align*}
This ends the proof of the claim.

\paragraph{Part II}

Since each $g_n$ is a distribution function, and by Part I above for each $x$, $(g_n(x))$ is a monotone sequence for large $n$ (decreasing for $x < a$ and increasing for $x > a$), hence the limit $g(x) = \lim_{n \rightarrow \infty} g_n(x)$ exists for all $x$. 
To show that for $x < a$, $g(x) = 0$, assume that for some $0 < x < a$, $g(x) = \varepsilon > 0$. This implies that 
\begin{align}
\tilde f^n \left(\frac{x}{n^{\frac{1}{k-1}}}\right) \geq \varepsilon > 0
\label{eq:gn_not_0}
\end{align}
for large $n$. Take $y \in \R,l\in \mathbb{N}$ such that $y = x \left( \frac{l}{l-1} \right)^{\frac{1}{k-1}} < a$. Note also that $g(y) \leq 1$, but since $(g_n(y))$ is a strictly decreasing sequence, the inequality is in fact sharp, thus
\begin{align}
\lim_{n \to \infty} g_n(y) < 1.
\label{eq:lim_gn_sharp}
\end{align} 
On the other hand,
\begin{align}
\begin{split}
\lim_{n \to \infty} g_n(y) = {} & \lim_{n \to \infty} g_{nl}(y) = \lim_{n \to \infty} \tilde f^{nl}\left(\frac{y}{(nl)^{\frac{1}{k-1}}}\right) \\
= {} & \lim_{n \to \infty} \tilde f^n \circ \tilde f^{n(l-1)} \left(\frac{y\left(\frac{l-1}{l}\right)^{\frac{1}{k-1}}}{(n(l-1))^{\frac{1}{k-1}}}\right) \\
= {} & \lim_{n \to \infty} \tilde f^n \circ \tilde f^{n(l-1)}\left( \frac{x}{(n(l-1))^{\frac{1}{k-1}}} \right), 
\end{split}
\label{eq:lim_gn_manipulations}
\end{align}
and by (\ref{eq:gn_not_0}),
\begin{align*}
\lim_{n \to \infty} \tilde f^n \circ \tilde f^{n(l-1)}\left(\frac{x}{(n(l-1))^{\frac{1}{k-1}}} \right) \geq \lim_{n \to \infty} \tilde f^n(\varepsilon) = 1,
\end{align*}
as $1$ is the only stable fixed point of $\tilde f$ in the interval $[0, 1]$. But this contradicts (\ref{eq:lim_gn_sharp}) and thus $g(x) = 0$.

Similarly, to show that for $x > a$, $g(x) = 1$, fix any such $x$ and take $y \in \R,l\in \mathbb{N}$ such that $y = \left( \frac{l-1}{l} \right)^{\frac{1}{k-1}} x > a$. By calculations similar to (\ref{eq:lim_gn_manipulations}), 
\begin{align*}
\begin{split}
\lim_{n \to \infty} g_n(x) = {} & \lim_{n \to \infty} g_{nl}(x) = \lim_{n \to \infty} \tilde f^{nl}\left(\frac{x}{(nl)^{\frac{1}{k-1}}}\right) \\
= {} & \lim_{n \to \infty} \tilde f^n \circ \tilde f^{n(l-1)} \left(\frac{y}{(n(l-1))^{\frac{1}{k-1}}}\right) \\
= {} & \lim_{n \to \infty} \tilde f^n \circ g_{n(l-1)}(y),
\end{split} 
\end{align*}
but since $(g_n(y))$ is a strictly increasing sequence for large $n$, for these $n$, $g_n(y) \geq \delta$ for some $\delta > 0$, hence
\begin{align*}
\lim_{n \to \infty} \tilde f^n \circ g_{n(l-1)}(y) \geq  \lim_{n \to \infty} \tilde f^n(\delta) = 1,
\end{align*}
as again, $1$ is the only stable fixed point of $\tilde f$ in $[0,1]$.

Recall now the definition of $g_n(x)$ (\ref{eq:gn_f_1}). Parts I and II prove that
\begin{align*}
\mathcal{L} \left( n^{\frac{1}{k-1}}\frac{W_{2n}-q}{q_+-q} \ \bigg| \ W_{2n} \in [q, q_+] \right)\stackrel{}{\to} \delta_a,
\end{align*}
and thus
\begin{align*}
\mathcal{L} \left( n^{\frac{1}{k-1}}(W_{2n}-q) \ \bigg| \ W_{2n} \in [q, q_+] \right)\stackrel{}{\to} \delta_{ a(q_+-q)},
\end{align*}
where
\begin{align*}
a(q_+-q) = \left( \frac{k(k-2)!}{f^{(k)}(q)} \right)^{\frac{1}{k-1}}.
\end{align*}
This completes the proof of Lemma \ref{lm:f_1_onesided}.
\end{proof}

To finish the proof of part (\ref{thm:main_eq_1}) of Theorem \ref{prof:prop_main} we apply Lemma \ref{lm:f_1_onesided} and its counterpart for points unstable from the left to $[q,q_+]$ and $[q_-,q]$ respectively. Checking that for each $n$
\begin{align*}
P(W_{2n} \in [q, q_+] \ | \ W_{2n} \in [q_-,q_+]) = {} & \frac{q_+-q}{q_+-q_-},\\
P(W_{2n} \in [q_-, q] \ | \ W_{2n} \in [q_-,q_+]) = {} & \frac{q-q_-}{q_+-q_-},
\end{align*}
shows that the masses in the formulation of the theorem are chosen appropriately, hence ends the proof. 
\subsection{Proof of Theorem 
\ref{prof:prop_main}(\ref{thm:main_eq_inf}): \texorpdfstring{$f'(q) = \infty$}{f'(q) = infinity}}
Note first that $f'(q)= \infty$ can happen only at 
$q \in \{0, 1\}$. 

We start by describing behaviour of $f$ near $0$. The first step is supplied by the following technical lemma:

\begin{lemma} There exist functions $H(x)$ and $b(x)$ defined on $(-1,1)$ such that $f(x) \sim H(x)$ as $x \to 0$ and 
\begin{align*}
H'(x) \sim \left(\sum_{n=1}^\infty n p_n b(x)^{n-1}\right) K p_K x^{K-1},
\end{align*}
where $1 - b(x) \sim p_K x^K$ as $x \to 0$.
\label{lm:tractable_H}
\end{lemma}
\begin{proof}[Proof of Lemma \ref{lm:tractable_H}]
By simple calculations,
\begin{align}
\begin{split}
f(x) = {} & R(R(x))= \frac{R(R(x))}{1-R(x)}(1-R(x)) = \frac{1-\sum_{n=1}^\infty p_n R(x)^n}{1-R(x)} \sum_{n=1}^\infty p_n x^n \\
= {} & \frac{\sum_{n=1}^\infty p_n (1-R(x)^n)}{1-R(x)} \sum_{n=1}^\infty p_n x^{n} = \sum_{n=1}^\infty \left( p_n \sum_{i=0}^{n-1}R(x)^i \right) \sum_{n=1}^\infty p_nx^{n} \\
= {} & \sum_{i=0}^\infty \left (R(x)^i \sum_{n=i+1}^\infty p_n\right) \sum_{n=1}^\infty p_nx^{n} = \sum_{i=0}^\infty \left[ R(x)^i \P (M > i) \right]  \sum_{n=1}^\infty p_nx^{n},
\end{split}
\label{eq:RRx_div_x}
\end{align}
where $M$ is a random variable with law $\P(M=i) = p_i$. Furthermore, recalling that $K= \min\{i : p_i \neq 0\}$,
\begin{align}
\begin{split}
\frac{R(x)}{1-x} = {} & \frac{\sum_{n=1}^\infty p_n(1-x^n)}{1-x} = \sum_{n=1}^\infty p_n \sum_{i=0}^{n-1} x^{i} = \sum_{i=0}^\infty x^i \sum_{n=i+1}^\infty p_n \\
= {} &  \sum_{i=0}^\infty x^i \P(M > i) = 1 + x+\ldots+x^{K-1}+x^K\sum_{i=K}^\infty x^{i-K} \P(M > i).
\end{split}
\label{eq:Rx_div_1-x}
\end{align}
Therefore, substituting (\ref{eq:Rx_div_1-x}) into (\ref{eq:RRx_div_x}),
\begin{align}
f(x) = \sum_{i=0}^\infty \left[(1-x)(1+x+\ldots+x^{K-1}+x^Kh(x))\right]^i \P (M > i)  \sum_{n=1}^\infty p_nx^{n},
\label{eq:RRx_div_x_2}
\end{align}
where 
\begin{align}
h(x)=\sum_{i=K}^\infty x^{i-K} \P(M > i) \to \P(M > K) = 1 - p_K \quad \textrm{ as } x \to 0.
\label{eq:limit_h}
\end{align}
Observe that $h'(x) \to \P(M > K+1)$ as $x \to 0$. Now for any $b < 1$,
\begin{align*}
\sum_{i=0}^\infty b^i \P(M > i) = \sum_{n=1}^\infty p_n \sum_{i=0}^{n-1} b^i = \frac{1}{1-b}\left(1 - \sum_{n=1}^\infty p_n b^n \right),
\end{align*}
and thus, setting 
\begin{align*}
b(x) = (1-x)(1+x+\ldots+x^{K-1}+x^K h(x)) = 1 - x^K + x^K h(x) - x^{K+1} h(x),
\end{align*}
from (\ref{eq:RRx_div_x_2}) we obtain
\begin{align}
f(x) = \frac{1}{1-b(x)}\left(1 - \sum_{n=1}^\infty p_n b(x)^n  \right)\sum_{n=1}^\infty p_nx^{n}.
\label{eq:RRx_div_x_3}
\end{align}
Observe that
\begin{align*}
1-b(x) = x^K(1-h(x)) + x^{K+1} h(x),
\end{align*} 
hence by (\ref{eq:limit_h}),
\begin{align*}
1 - b(x) \sim p_K x^K \quad \textrm{ as } x \to 0.
\end{align*}
Moreover, from (\ref{eq:RRx_div_x_3}),
\begin{align*}
f(x) = {} & \frac{1}{x^K(1-h(x)) + x^{K+1}h(x)}\left(1 - \sum_{n=1}^\infty p_n b(x)^n  \right)\sum_{n=1}^\infty p_nx^{n} = \\
= {} & \frac{1}{1-h(x) + x h(x)}\left(1 - \sum_{n=1}^\infty p_n b(x)^n  \right)\sum_{n=K}^\infty p_nx^{n-K}.
\end{align*}
Note that as $x\to0$, 
the first fraction on the right-hand side 
converges to $\frac{1}{p_K}$, and the final sum
converges to $p_K$. 
Thus
\begin{align*}
f(x)\sim {} & \left(1 - \sum_{n=1}^\infty p_n b(x)^n  \right)
\end{align*}
as $x \to 0$. Therefore, denoting $H(x) := \left(1 - \sum_{n=1}^\infty p_n b(x)^n  \right)$, we have 
\begin{align}
f(x) \sim {} &  H(x)
\label{eq:f_leading_term}
\end{align}
and, recalling that $\lim_{x \to 0}$ and $\lim_{x \to 0} h'(x)=\P(M > K+1)$,
\begin{align*}
H'(x) = {} &  \left(1 - \sum_{n=1}^\infty p_n b(x)^n\right)' \sim \left(\sum_{n=1}^\infty n p_n b(x)^{n-1}\right) K p_K x^{K-1}
\end{align*}
as $x \to 0^+$ which ends the proof of the lemma.
\end{proof}

Equipped with the relation from Lemma \ref{lm:tractable_H} we may now connect $f$ with the underlying offspring distribution via \textit{Karamata's Tauberian Theorem for Power Series} (a proof may be found e.g.\ in \cite{Bingham1987}). Recall first the theorem:
\begin{theorem}[Karamata's Tauberian Theorem] %p. 40 1.7.4
If $a_n \geq 0$ and the power series $A(s) = \sum_{n=0}^\infty a_n s^n$ converges for $s \in [0,1)$, then for $c, \rho \geq 0$ 
%and a slowly varying function $l$ 
the following are equivalent:
\begin{align*}
% version with slowly varying function: \sum_{k=0}^n a_k \sim % c n^\rho l(n)/\Gamma(1+\rho) \quad (n \to \infty)
\sum_{k=0}^n a_k \sim c n^\rho \text{ as } n \to \infty
\end{align*}
and
\begin{align*}
% version with slowly varying function: A(s) \sim cl(1/(1-s))/
% (1-s)^\rho \quad(s \uparrow 1).
A(s) \sim \frac{c\Gamma(1+\rho)}{(1-s)^\rho} \text{ as } s \uparrow 1.
\end{align*}
\label{thm:karamata}
\end{theorem}
Recall the assumption (\ref{assumption:mean}) of Theorem \ref{prof:prop_main}: for some $\rho \in (0,1)$,
\begin{align*}
\mathbb{E}(M \mathbb{I}_{M \leq n}) =  \sum_{k=1}^n k p_k \sim c n^\rho .
\end{align*}

By Theorem \ref{thm:karamata} applied to $a_k=kp_k$ and (\ref{assumption:mean}) we obtain that as we let $x \to 0$ (which implies $b(x) \to 1$),
\begin{align*}
\frac{1}{K p_K x^{K-1}} H'(x) \sim {} & \sum_{n=1}^\infty n p_n b(x)^{n-1} \sim  \frac{c\Gamma(1+\rho)}{(1-b(x))^\rho} \sim \frac{c\Gamma(1+\rho)}{p_K^\rho} \frac{1}{x^{K\rho} }.
\end{align*}
Therefore,
\begin{align}
H(t) = \int_0^t H'(x) \textrm{d} x \sim \int_0^t  \frac{c \Gamma(1+\rho)}{p_K^\rho} \frac{1}{x^{K\rho}} K p_K x^{K-1}\textrm{d} x = \frac{c \Gamma(1+\rho) p_K^{1-\rho}}{1-\rho} t^{K-K\rho},
\label{eq:H_explicit}
\end{align}
as $t \to 0$, hence, by (\ref{eq:f_leading_term}) and (\ref{eq:H_explicit}),
\begin{align*}
f(t) \sim \frac{c \Gamma(1+\rho) p_K^{1-\rho}}{1-\rho} t^{K(1-\rho)}
%\label{eq:asymptotics}
\end{align*}
as $t \to 0$. This implies that  for $f'(0)=\infty$ to hold it is necessary that $K < \frac{1}{1-\rho}$.

To provide the criterion for $q=1$, we are interested in the behaviour of the quantity $1 - f(t)$ when $t \to 1$. Now
\begin{align}
1-f(t) = 1 - R(R(t)) = G(R(t)) \sim p_K R(t)^K.
\label{eq:1-f}
\end{align}
By definition, $R(x) = 1 - \sum_{k=1}^\infty p_k x^k$, thus
\begin{align*}
R'(x) = - \sum_{k=1}^\infty k p_k x^{k-1},
\end{align*}
and again by Theorem \ref{thm:karamata} applied to $a_k = kp_k$ and (\ref{assumption:mean}),
\begin{align*}
R'(x) \sim - \frac{c \Gamma(1+\rho)}{(1-x)^\rho}
\end{align*}
as $x \to 1$, and thus
\begin{align}
R(t) = R(1) - \int_t^1 R'(x)\textrm{d}x = \frac{c \Gamma(1+\rho)}{1-\rho}\left( 1-t\right)^{1-\rho}.
\label{eq:R(1)}
\end{align}
Substituting (\ref{eq:R(1)}) into (\ref{eq:1-f}) we obtain that
\begin{align*}
1 - f(t) \sim p_K\left(\frac{c \Gamma(1+\rho)}{1-\rho}\right)^K \left( 1-t\right)^{K(1-\rho)}
%\label{eq:1-f_asymptotics}
\end{align*}
as $t \to 1$. Thus again, for $ f'(1) = \infty$ to hold it is necessary that $K(1-\rho) < 1$.

We've shown that $f(t) \sim C_0 t^{K(1-\rho)}$ as $t \to 0$ and $1-f(t) \sim C_1 (1-t)^{K(1-\rho)}$ as $t \to 1$ for some positive constants $C_0, C_1$ that we determined explicitly. Note that
\begin{align*}
C_1 = C_0^k p_K^{1-K(1-\rho)}
\end{align*}
and since $K(1-\rho) < 1$, at least one of the constants $C_0, C_1$ is different from $1$. The proof of Proposition 
\ref{prof:prop_main}(\ref{thm:main_eq_inf}) is completed by the following two lemmas applied as follows: Lemma \ref{prop:nonlinearl_scaling} applied with $\alpha = K(1-\rho)$ proves existence of the distributional limit at either point $q$ with $C_q \neq 1$, and Lemma \ref{prop:limit_iff} shows that the limit exists at $q=0$ if and only if it exists at $q=1$.

\begin{lemma}
Consider the recursion (\ref{quantilerecursion});
\begin{enumerate}
\item Assume that $f(t) \sim C t^{\alpha}$ with $C \neq 1$ and $\alpha \in (0,1)$ as $t \to 0$. Let $q_+ = \inf\{x: x>0, \ x=f(x)\}$. Then
\begin{align}
\mathcal{L} \left( \alpha^n \log W_{2n} \ | \ W_{2n} \in [0, q_+] \right) \xrightarrow[]{d} V_0,
\label{eq:claim_nonlinear_sc1}
\end{align}
where $V_0$ is a random variable with $\P(V_0 \in (- \infty, 0))=1$.
\item Assume that $1-f(t) \sim C (1-t)^{\alpha}$ with $C \neq 1$ and $\alpha \in (0,1)$ as $t \to 1$. Let $q_-=\sup\{x: x<1, \ x=f(x)\}$. Then
\begin{align}
\mathcal{L} \left( \alpha^n \log (1-W_{2n}) \ | \ W_{2n} \in [q_-, 1] \right) \xrightarrow[]{d} V_1,
\label{eq:claim_nonlinear_sc2}
\end{align}
where $V_1$ is a random variable with $\P(V_1 \in (- \infty, 0))=1$.
\end{enumerate}
\label{prop:nonlinearl_scaling}
\end{lemma}
\begin{lemma}
Consider the recursion (\ref{eq:rde}). For $\alpha \in (0,1)$ convergence (\ref{eq:claim_nonlinear_sc1}) holds for some $V_0$ with $\P(V_0 \in (- \infty, 0))=1$ if and only if convergence (\ref{eq:claim_nonlinear_sc2}) holds for some $V_1$ with $\P(V_1 \in (- \infty, 0))=1$.
\label{prop:limit_iff}
\end{lemma}
Note that in Lemma \ref{prop:limit_iff} we do not assume anything 
about $f$; in particular we do not assume that $C \neq 1$.

%\marginpar{$C\neq 1$ above and below}
\begin{proof}[Proof of Lemma \ref{prop:nonlinearl_scaling}]
Firstly we show how the second part can be obtained from the first one and then we prove the first part of Lemma \ref{prop:nonlinearl_scaling}, which corresponds to $q=0$.

Assume that $1-f(t) \sim C(1-t)^\alpha$ and set $\tilde W_{2n} = 1 - W_{2n}$ and $\tilde f(t) = 1 - f (1 - t)$. Then,
\begin{align*}
\P(\tilde W_{2n} \leq x) = {} & \P(1 - W_{2n} \leq x ) = 1 - \P(W_{2n} \leq 1 -x) =  1 - f(\P(W_{2n-2} \leq 1 -x)) \\
= {} & 1 - f(1 - \P(\tilde W_{2n-2} \leq x)) = \tilde f (\P(\tilde W_{2n-2} \leq x)) = \ldots = f^n (\P(\tilde W_0 \leq x)) \\
= {} & \tilde f^n(x)
\end{align*}
and $\tilde f (t) = 1 - f(1-t) \sim C t^\alpha$ as $t \to 0$. 
Hence it is enough to prove the result for the case $q=0$.

Fix some $x < 0$. Note that for $n$ large enough $\exp \left(\frac{x}{\alpha^n}\right) \leq q_+$, hence for these $n$,
\begin{align}
\begin{split}
\P(\alpha^n \log W_{2n} \leq x \ | \ W_{2n} \in [0, q_+]) = {} & \P \left(W_{2n} \leq \exp\left( \frac{x}{\alpha^n}\right) \ \Big| \ W_{2n} \in [0, q_+]\right) = \\
= {} & \frac{1}{q_+} f^{n} \left( \exp \left(\frac{x}{\alpha^n}\right)\right).
\end{split}
\label{eq:exponential_f}
\end{align}
Define
\begin{align*}
g (y)= \log ( f ( \exp(y))),
\end{align*}
and observe that
\begin{align}
g^n (y)= \log ( f^n ( \exp(y))).
\label{eq:def_g}
\end{align}
The idea behind $g(y)$ is to linearize $f(y)$: note that $g$ is a monotone function, $g(\log q_+) = \log q_+$ and that
\begin{align*}
g(y) = \alpha y + O(1)
\end{align*}
as $y \to -\infty$, hence there exist constants $\tilde D, \tilde E$ such that for $y \leq \log q_+ < 0$,
\begin{align}
\tilde D + \alpha y \leq g(y) \leq \tilde E + \alpha y.
\label{eq:bound_g}
\end{align}

\begin{center}
\begin{figure}[h]
\begin{tikzpicture}
    % Draw axes
  \draw[->] (-3,0) node (xaxis) {} -- (2,0) node (xaxis2) [below right] {$x$};
  \draw[->] (0,-2.5) node (yaxis) {}  -- (0,2) node (yaxis2) [above right] {$y$};
  
  \draw (-3,-2.1) coordinate (a_1) -- (1.8,0.3) coordinate (a_2) node [above right] {$h_1(x)$};
  \draw (-3,-1.1) coordinate (b_1) -- (1.8,1.3) coordinate (b_2) node [above right] {$h_2(x)$};
  \draw[thin, dashed] (-2,-2) coordinate (c_1) -- (1.5,1.5) coordinate (c_2);
  
  \coordinate (c1) at (intersection of a_1--a_2 and c_1--c_2);
  \coordinate (c2) at (intersection of b_1--b_2 and c_1--c_2);
  \coordinate (logq) at (-0.5,-0.5);
  
  \draw[dashed] (yaxis |- logq) node[right] {$\log q_+$}
      -| (xaxis -| logq) node[below] {};
  \draw[dashed] (yaxis |- c1) node[right] {$\frac{\tilde D}{1-\alpha}$}
      -| (xaxis -| c1) node[below] {}; 
  \draw[dashed] (yaxis |- c2) node[left] {$\frac{\tilde E}{1-\alpha}$}
      -| (xaxis -| c2) node[below] {};

  \fill[black] (logq) circle (2pt);  
  \fill[black] (c1) circle (2pt);
  \fill[black] (c2) circle (2pt);
  
  \draw plot[smooth] coordinates {(-3,-1.5) (-2.5,-1.2) (-1.5, -0.8) (-0.5,-0.5)};
  \draw (-3,-1.5) node[left] {$g(x)$};
 
\end{tikzpicture}
\caption{
\label{fig:linearization}
$h_1(x), g(x), h_2(x)$ together with their (stable) fixed points $\frac{\tilde D}{1-\alpha}, \log q_+, \frac{\tilde E}{1-\alpha}$ respectively. For $x \leq \log q_+$, $h_1(x) \leq g(x) \leq h_2(x)$. The dashed line represents the identity function.}
\end{figure}
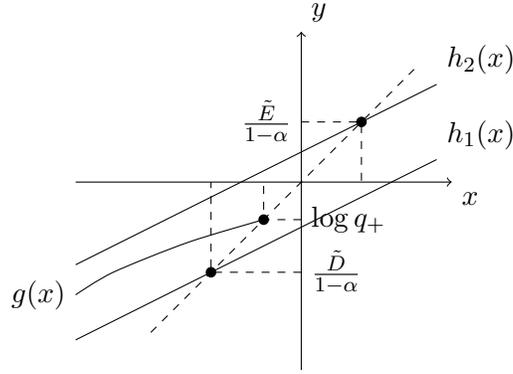
\end{center}

In the first part of the proof we show (assuming that the limit (\ref{eq:claim_nonlinear_sc1}) exists) that $\P(V \in (-\infty,0)) = 1$. Define
\begin{align*}
h_1(y) = {} & \tilde D + \alpha y, \\
h_2(y) = {} & \tilde E + \alpha y.
\end{align*}
Now
\begin{align}
\begin{split}
\lim_{n \to \infty} h_1^{n}\left(\frac{y}{\alpha^n}\right) = {} & y + \frac{\tilde D}{1-\alpha},\\
\lim_{n \to \infty} h_2^{n}\left(\frac{y}{\alpha^n}\right) = {} & y + \frac{\tilde E}{1-\alpha},
\end{split}
\label{eq:conv_h_1}
\end{align}
where $\frac{\tilde D}{1-\alpha}$, $\frac{\tilde E}{1-\alpha}$ are the (unique) fixed points of $h_1$ and $h_2$ respectively. 

Equations (\ref{eq:exponential_f}), (\ref{eq:def_g}), (\ref{eq:bound_g}) and (\ref{eq:conv_h_1}) together imply that 
\begin{align*}
%\frac{1}{q_+} \exp\left(x + \frac{\tilde D}{\alpha}\right) {} & \leq \liminf_{n \to \infty} \P(\alpha^n \log W_{2n} \leq x | W_{2n} \in [0, q_+] ) \leq \\
\limsup_{n \to \infty} \P(\alpha^n \log W_{2n} \leq x | W_{2n} \in [0, q_+] ) \leq \frac{1}{q_+} \exp\left( x + \frac{\tilde E}{1-\alpha}\right),
\end{align*}
and therefore
\begin{align*}
\lim_{x \to -\infty} \limsup_{n \to \infty} \P(\alpha^n \log W_{2n} \leq x \ | \ W_{2n} \in [0, q_+]) = 0.
\end{align*}
We shall now show that 
\begin{align}
\lim_{x \to 0^-} \liminf_{n \to \infty} \P(\alpha^n \log W_{2n} \leq x \ | \ W_{2n} \in [0, q_+]) = 1.
\label{eq:nonlinear_upper}
\end{align}
To do so, note first that by (\ref{eq:exponential_f}) and (\ref{eq:def_g})
\begin{align*}
\P(\alpha^n \log W_{2n} \leq x \ | \ W_{2n} \in [0, q_+]) = \frac{1}{q_+} \exp \left( g^n\left(\frac{x}{\alpha^n}\right)\right),
\end{align*}
hence (\ref{eq:nonlinear_upper}) is equivalent to
\begin{align}
\lim_{x \to 0^-} \liminf_{n \to \infty} g^n\left(\frac{x}{\alpha^n}\right) = \log q_+.
\label{eq:contradiction_g}
\end{align}
Note that $\log q_+$ is a fixed point of $g(x)$. (\ref{eq:contradiction_g}) indicates that the scaling $\alpha^n$ is not strong enough to compensate the attraction of the fixed point $\log q_+$ of $g$. 

Let $k_{x,n}$ be the smallest $k$ such that $h_1^k \left(\frac{x}{\alpha^n}\right) \geq \frac{\tilde D}{1-\alpha} - 1$. Note that $k_{x,n}$ is properly defined as $\frac{\bar D}{1-\alpha}$ is the only fixed point of $h_1$ and is stable. Moreover, by (\ref{eq:conv_h_1}) we have that for $x \in (-1,0)$,
\begin{align*}
\lim_{n \to \infty} h_1^n \left(\frac{x}{\alpha^n}\right) = x + \frac{\bar D}{1-\alpha} \geq \frac{\tilde D}{1-\alpha} - 1,
\end{align*}
hence for these $x$, $n - k_{x,n} \geq 0$ for large $n$. Define also
\begin{align*}
K_{x} = \liminf_{n \to \infty} (n - k_{x,n})
\end{align*}
and note that since
\begin{align*}
\lim_{x \to 0^-} \lim_{n \to \infty} h_1^{n}\left(\frac{x}{\alpha^n}\right) = {} &\frac{\tilde D}{1-\alpha},
\end{align*}
and the right-hand side is a fixed point of $h_1$, we obtain that
\begin{align*}
\lim_{x \to 0^-} K_{x} = \infty.
\end{align*}

Now define similarly $\tilde k_{x,n}$ to be the smallest $k$ such that $g^k \left(\frac{x}{\alpha^n}\right) > \frac{\tilde D}{1-\alpha} - 1$. Since $g(y) \geq h_1(y)$ for $y \leq \log q_+$, we have $k_{x,n} \geq \tilde k_{x,n}$, and therefore
\begin{align*}
\lim_{x \to 0^-} \liminf_{n \to \infty} (n - \tilde k_{x,n}) = \infty.
\end{align*}
This implies that
\begin{align*}
\lim_{x \to 0^-} \liminf_{n \to \infty} g^n\left(\frac{x}{\alpha^n}\right) = {} &  \lim_{x \to 0^-} \liminf_{n \to \infty} g^{n-k_{x,n}} \left(g^{k_{x,n}}\left(\frac{x}{\alpha^n}\right)\right) \\
\geq {} & \lim_{x \to 0^-} \liminf_{n \to \infty} g^{n-k_{x,n}}\left(\frac{\tilde D}{1-\alpha}-1\right) \\
= {} & \log q_+.
\end{align*}

To finish the proof it is now enough to justify that the limit $\lim_{n \to \infty} \P(\alpha^n \log W_{2n} \leq x | W_{2n} \in [0, q_+] )$ exists for all $x$; recalling (\ref{eq:exponential_f}) it is enough to check that the sequence $f^{n}\left( \exp \left(\frac{x}{\alpha^n}\right)\right)$ is monotone for large $n$. Since $f$ is a  strictly monotone function, the statements
\begin{align*}
f^{n+1} \left( \exp \left(\frac{x}{\alpha^{n+1}}\right) \right) {} & \geq f^{n} \left( \exp \left(\frac{x}{\alpha^{n}}\right)\right)
\end{align*}
and
\begin{align}
f \left( \exp \left(\frac{x}{\alpha^{n+1}}\right)\right) {} & \geq \exp \left( \frac{x}{\alpha^{n}} \right).
\label{eq:nonlinear_fn_monotone}
\end{align}
are equivalent. We set $y = \frac{x}{\alpha^{n+1}}$ and $z = \exp(y)$ (therefore $y \to -\infty$ corresponds to $z \to 0$) obtaining that (\ref{eq:nonlinear_fn_monotone}) is equivalent to:
\begin{align*}
f(z) {} & \geq z^{\alpha}.
\end{align*}
Therefore, if $f(z) \sim C z^{\alpha}$ for $C \neq 1$ 
%\marginpar{what if $f(z) \sim z^{\alpha}$?}
we observe that for each $x<\log q_+$ the sequence $f^{n}\left(\exp\left(\frac{x}{(\alpha)^n}\right)\right)$ is monotone for $n$ large enough which yields existence of the limit. 
This ends the proof of Lemma \ref{prop:nonlinearl_scaling}.
\end{proof}
\begin{comment}
\begin{remark}
Denoting by $\psi(x)$ the distribution function of the limiting random variable $W$, recalling the recursion identity (\ref{eq:recursion_prob}) we obtain that $\psi(x)$ solves the following fixed point equation:
\begin{align*}
\psi (x) = \lim_{n \to \infty} \P(\alpha^n \log W_{2n} \leq x \ | \ W_{2n} \in [0, q_+]) = \frac{1}{q_+} f \left( q_+ \psi \left( \frac{x}{\alpha} \right)\right).
\end{align*}
\end{remark}
\end{comment}

Up to now we only defined $W_m$ for even $m$. 
Before we prove Lemma \ref{prop:limit_iff} we extend
to odd $m$. 
Define the distribution of a random variable $W_{2n-1}$ as follows:
\begin{align*}
W_{2n-1} \stackrel{d}{=} \max_{1 \leq i \leq M} W_{2n-2}^{(i)},
\end{align*}
where $M$ is a random variable from drawn the tree's offspring distribution and $W_{2n-2}^{(i)}$ are independent copies of $W_{2n-2}$
(independent of $M$). The quantity $W_{2n-1}$ corresponds to the value
at the root of a height $2n-1$, with levels alternating
between max and min, starting and ending with a max. 
One has similarly
\begin{align*}
W_{2n} \stackrel{d}{=} \min_{1 \leq i \leq M} W_{2n-1}^{(i)}.
\end{align*}
Lemma \ref{lm:swap} below provides a useful identity which we are going to apply in the proof of Lemma \ref{prop:limit_iff}.

\begin{lemma}
% Let $G(t)$ be the moment generating function associated with the %tree's offspring distribution. Then
%\begin{align*}
$
W_{2n-1} \stackrel{d}{=} G^{-1}(1-W_{2n-2}).
$
%\label{eq:consecutive_levels}
%\end{align*}
\label{lm:swap}
\end{lemma}
\begin{proof}[Proof of Lemma \ref{lm:swap}]

$G$ is the probability generating function of the offspring 
distribution of the tree, so 
$G(t) = \P(\max_{1\leq i \leq M} U_i \leq t)$ where $U_i$ are independent uniform random variables and $M$ follows the offspring distribution (independently of $(U_i, i\geq 1)$). Decomposing the minimax tree of height $2n-1$ with maximum at levels $1$ and $2n-1$, we see that random variables at level $2n-2$ (i.e. one level above the leaves) are distributed as $\max_{1\leq i \leq M} U_i$. Therefore
\begin{align}
W_{2n-1} \stackrel{d}{=} W_{2n-2}^{\text{max}, G},
\label{eq:swap_proof_1}
\end{align}
where $W_{2n-2}^{\text{max}, G}$ is a random variable corresponding to a max-min tree (i.e. with maximum at the even levels and minimum at the odd ones) where at the leaves instead of uniform random variables we put random variables with distribution function $G$. Noting that if $U$ is a uniform random variable then $G^{-1}(U)$ has distribution function $G$, we see that 
\begin{align}
W_{2n-2}^{\text{max}, G} \stackrel{d}{=} G^{-1}(W_{2n-2}^{\text{max}}).
\label{eq:swap_proof_2}
\end{align}
Now, since
\begin{align*}
\max_{1\leq i \leq M}U_i = 1 - \min_{1 \leq i \leq M} (1-U_i) \stackrel{d}{=}1 - \min_{1 \leq i \leq M} U_i
\end{align*}
and
\begin{align*}
\min_{1\leq i \leq M}U_i = 1 - \max_{1 \leq i \leq M} (1-U_i) \stackrel{d}{=}1 - \max_{1 \leq i \leq M} U_i,
\end{align*}
we obtain that 
\begin{align}
W_{2n-2}^{\text{max}} \stackrel{d}{=} 1 - W_{2n-2}.
\label{eq:swap_proof_3}
\end{align}
Finally, combining (\ref{eq:swap_proof_1}), (\ref{eq:swap_proof_2}) and (\ref{eq:swap_proof_3}) completes the proof.
\end{proof}
We are now ready to prove Lemma \ref{prop:limit_iff}.
\begin{proof}[Proof of Lemma \ref{prop:limit_iff}]
\begin{comment}
Without loss of generality we are going to show how the convergence
\begin{align}
\mathcal{L} \left( \alpha^n \log (1 - W_{2n}) \ | \ W_{2n} \in [q_-, 1] \right) \xrightarrow[]{d} V_1,
\label{eq:equiv_assump}
\end{align}
implies the convergence 
\begin{align}
\mathcal{L} \left( \alpha^n \log W_{2n} \ | \ W_{2n} \in [0, q_+] \right) \xrightarrow[]{d} V_0.
\end{align}
\end{comment}
The convergence (\ref{eq:claim_nonlinear_sc1}) is equivalent to the convergence of
\begin{align}
\lim_{n \to \infty} \P(\alpha^n \log W_{2n} \leq x \ | \ W_{2n} \in [0, q_+]).
\label{eq:conv_equiv1}
\end{align}
at all the continuity points of the corresponding limiting distribution function and similarly the convergence (\ref{eq:claim_nonlinear_sc2}) is equivalent to the convergence of
\begin{align}
\lim_{n \to \infty} \P(\alpha^n \log (1-W_{2n}) \leq x \ | \ W_{2n} \in [q_-,1]).
\label{eq:conv_equiv1'}
\end{align}
at all the continuity points of the corresponding limiting distribution function. Fix $x < 0$. For large $n$,
\begin{align*}
\P(\alpha^n \log W_{2n} \leq x \ | \ W_{2n} \in [0, q_+]) = {} & \frac{1}{q_+}\P(\alpha^n \log W_{2n} \leq x , \ W_{2n} \in [0, q_+]) \\
= {} & \frac{1}{q_+}\P(W_{2n} \leq \exp (x/\alpha^n) , \ W_{2n} \in [0, q_+]) \\
= {} & \frac{1}{q_+} \P(W_{2n} \leq \exp (x/\alpha^n)).
\end{align*}
Now by the branching structure of the tree,
\begin{align*}
\P(W_{2n} \leq \exp (x/\alpha^n)) = 1 - G(\P(W_{2n-1} > \exp (x/\alpha^n))).
\end{align*}
Since $G$ is a continuous function, the convergence (\ref{eq:conv_equiv1}) is equivalent to the convergence
\begin{align*}
\lim_{n \to \infty} \P(W_{2n-1} > \exp (x/\alpha^n)).
\end{align*}
By Lemma \ref{lm:swap}, 
\begin{align*}
\begin{split}
\P(W_{2n-1} > \exp (x/\alpha^n)) = {} & \P(G^{-1}(1-W_{2n-2}) > \exp (x/\alpha^n)) \\
= {} & \P(1-W_{2n-2} > G(\exp (x/\alpha^n))) \\
= {} & \P(\alpha^{n-2} \log (1-W_{2n-2}) > \alpha^{n-2} \log (G(\exp (x/\alpha^n))))\\
= {} & 1 - \P(\alpha^{n-2} \log (1-W_{2n-2}) \leq \alpha^{n-2} \log (G(\exp (x/\alpha^n)))).
\end{split}
\end{align*}
Since $G(t) \sim p_K t^K$ as $t \to 0$, we observe that
\begin{align*}
\alpha^{n-2} \log (G(\exp (x/\alpha^n))) = \alpha^{n-2} \log (p_K \exp((xK)/\alpha^n) + o(1) = \frac{xK}{\alpha^2} + o(1).
\end{align*}
This implies that if the convergence (\ref{eq:conv_equiv1}) holds at some point $\frac{xK}{\alpha^2}$ which is a continuous point of the limiting distribution function, then the convergence (\ref{eq:conv_equiv1'}) holds at $x$. Similarly, if the convergence (\ref{eq:conv_equiv1'}) holds at some point $x$ which is a continuous point of the limiting distribution function, then the convergence (\ref{eq:conv_equiv1}) holds at $\frac{xK}{\alpha^2}$. Since the set of discontinuity points of any distribution function is at most countable, this ends the proof.
%Since we assumed that the convergence  holds, $\P(\alpha^{n-2} \log (1-W_{2n-2}) \leq \alpha^{n-2} \log (G(\exp (x/\alpha^n))))$ converges for all $x$ such that $\frac{xK}{\alpha^2}$ is a continuity point of the cumulative distribution function of $V_1$. Since $G$ is a continuous function, this proves that $\P(\alpha^n \log W_{2n} \leq x \ | \ W_{2n} \in [0, q_+])$ converges for $x$ outside of a countable set, hence ends the proof.
\end{proof}

\section{Proof of the endogeny result}
\label{sec:endogeny_proof}
To prove Theorem \ref{thm:endogeny} we use the 
idea of \textit{bivariate uniqueness} introduced 
by Aldous and Bandyopadhyay \cite{AldBan}.

Informally, the idea is as follows: suppose we allow \textit{two} values at each node. Each coordinate evolves separately,
according to the minimax recursions (and using the same
realisation of the tree structure). If we put 
bivariate values at the leaves of the tree,
we then get a bivariate value at the root of the tree.
Let us consider the moment the case where the values 
are discrete (as for the Bernoulli case in Theorem \ref{thm:endogeny}).
If the process is endogenous, and the tree is large, 
then with high probability the two components at the root
agree with each other. On the other hand, if the 
process is not endogenous, then the probability that 
they disagree stays bounded away from zero as the size of the 
tree goes to infinity, and in fact we can obtain 
a bivariate process on the infinite tree which
is two-periodic and non-degenerate (in the sense that the
two components are not identically the same).

To formalise this we rewrite some of the ideas 
around (\ref{eq:rde}) 
in new notation. 

Let $\mu$ be a distribution on $[0,1]$. 
We defined $T(\mu)$ be the distribution
of the LHS of (\ref{eq:rde}), 
given that the random variables $W_{2n-2}^{(i,j)}$ on the 
RHS of (\ref{eq:rde}) are i.i.d.\ with distribution $\mu$.

So $T$ is a map from $\mathcal{P}$ to $\mathcal{P}$, 
where $\mathcal{P}$ is the space of distributions on $[0,1]$. 
For Theorem \ref{thm:endogeny}
we assume that the 
Bernoulli($1-x$) distribution is a fixed point of $T$. 

Now consider the space $\cPtwo$ of distributions
on $[0,1]^2$. Define the map $\Ttwo$ 
from $\cPtwo$ to itself as follows. 
As before let $M$ and $M_1,M_2,\dots$ be i.i.d.\
draws from the offspring distribution. 
Let $(X^{i,j}_1, X^{i,j}_2)$, for each $i,j$, be
i.i.d.\ with distribution $\mutwo$
(and independent of $M$ and $\{M_i\}$). 
Then let $\Ttwo(\mutwo)$ be the distribution of
$(X_1, X_2)$, where  
\begin{align*}
X_1&=
\min_{1\leq i\leq M}
\max_{1\leq j\leq M_i}
X^{(i,j)}_{1},
\\
X_2&= 
\min_{1\leq i\leq M}
\max_{1\leq j\leq M_i}
X^{(i,j)}_{2}.
\end{align*}
Note particularly that the recursions for $X_1$ and $X_2$
use the \textit{same} realisation of $M$ and $\{M_i\}$. 

If $\mu\in\cP$ then we can define 
a \textit{diagonal} distribution $\mudiag$
on $\cPtwo$ by $\mudiag=\text{dist}(X,X)$
if $\mu=\text{dist}(X)$. 

If $\mu$ is a fixed point of $T$, then certainly $\mudiag$
is a fixed point of $\Ttwo$. The question is whether there
can be any fixed point of $\Ttwo$, whose marginals are equal
to $\mu$, and which is \textit{not}
of the form of the diagonal distribution $\mudiag$. 
Mach, Sturm and Swart \cite[Theorem 1]{MachSturmSwart}, refining
Aldous and Bandyopadhyay \cite[Theorem 11]{AldBan},
show that the recursive tree process is endogenous if and only if
there are no such non-degenerate bivariate fixed points
(i.e.\ if the ``bivariate uniqueness property" holds). 

\begin{proof}[Proof of Theorem \ref{thm:endogeny}]
We apply Theorem 1 of \cite{MachSturmSwart}
(or indeed Theorem 11 of \cite{AldBan}, since the 
additional technical condition relating to continuity 
of the operator $\Ttwo$ does in fact hold in this setting).
To prove our result it is enough to show that
the bivariate uniqueness property holds if and only if
$f'(x)\leq 1$.

Let us write $\mu$ for the Bernoulli($1-x$) distribution
on $\{0,1\}$. 
We look for a distribution 
$\mutwo$ on $\{0,1\}^2$ which is a fixed point of $\Ttwo$, 
and whose marginals
are both $\mu$, but which is not the diagonal distribution $\mudiag$.
Once these marginals are specified, we only need
to specify one further parameter, say 
$b=\mutwo(1,0)$, since then we can deduce
$\mutwo(1,1)=1-x-\mutwo(1,0)=1-x-b$,
and similarly $\mutwo(0,1)=b$ and $\mutwo(0,0)=x-b$.
Note $b\in[0,\min(x,1-x)]$.

To show that $\mutwo$ is a fixed point of $\Ttwo$,
again it suffices to check just one entry of
$\Ttwo(\mutwo)$. To look at this we can consider a 
random tree with two levels, with bivariate
marginals according to $\mutwo$ at level 
2 of the tree; we wish to see distribution $\mutwo$ again 
at the root. Then write also $\nutwo$ for the corresponding 
distribution of the marginals at level 1. 
Let us write $o$ for the root and $\iota$
for a typical level-1 node. 

So consider the probability of seeing values $(1,0)$ at the root. 
For this to happen, all children of the root must have $1$
in the first coordinate, but at least one child of the root must
have $0$ in the second coordinate. That is, 
all children have values $(1,0)$ or $(1,1)$,
but not all of them have values $(1,1)$. 
We obtain 
\begin{align}
%b&=
\nonumber
\P\big(\text{values} (1,0) \text{ at }o\big)
%\\
&=G\big(\nutwo(1,0)+\nutwo(1,1)\big)-G\big(\nutwo(1,1)\big)\\
&=R\big(\nutwo(1,1)\big)-R\big(\nutwo(1,0)+\nutwo(1,1)\big).
\label{P10root}
\end{align}
We examine both the terms on the RHS. 
First note that $\nutwo(1,0)+\nutwo(1,1)$ is
the probability that $\iota$ has value $1$ in the first 
coordinate. This is the probability that at least one child
of $\iota$ has value 1 in the first coordinate,
i.e.\ that not all the children of $\iota$ have value
0 in the first coordinate. Hence

\begin{align}
\nonumber
\nutwo(1,0)+\nutwo(1,1)&=1-G(\mu(0,1)+\mu(0,0))\\
\nonumber
&=1-G(x)\\
&=R(x).
\label{nu1star}
\end{align}
Similarly, for $\iota$ to have values $(1,1)$,
we need to exclude the two events that
all its children have value $0$ in the first coordinate
or that all its children have value $0$ in the 
second coordinate. Both of these events have probability $G(x)$,
while their intersection, i.e.\ that all children have values
$(0,0)$, has probability $G(x-b)$. So applying inclusion-exclusion,
\begin{align}
\nonumber
\nutwo(1,1)&=1-G(x)-G(x)+G(x-b)\\
\label{nu11}
&=2R(x)-R(x-b).
\end{align}

Combining (\ref{P10root}), (\ref{nu1star}) and (\ref{nu11}),
we have that if the probability of values $(1,0)$ at level 2
is $b\in[0,\min(x,1-x)]$, then the probability of values $(1,0)$ at the root is 
$h(b)\in[0,\min(x,1-x)]$, where
\begin{equation}
\label{hdef}
h(b):=R(2R(x)-R(x-b))-R(R(x)).
\end{equation}
For $\mutwo$ to be a fixed point of $\Ttwo$, we therefore
need $b=h(b)$. 
Also $\mutwo$ is diagonal iff $b=0$. 
So non-endogeny is equivalent to the existence of a
fixed point of $h$ in the interval 
 $(0,\min(x,1-x)]$.

From (\ref{hdef})
we have $h(0)=0$, and differentiating with respect to $b$ we get
\begin{align}
h'(b)=R'(R(x)-[R(x-b)-R(x)])R'(x-b)
\label{h1fact}
\end{align}
so that
\begin{align*}
\begin{split}
h'(0)&=R'(R(x))R'(x)\\
&=\frac{d}{dx}R(R(x))\\
&=f'(x).
\end{split}
\end{align*}
Differentiating once more we obtain
\begin{equation}
\label{h2fact}
h''(b)=R''\big(2R(x)-R(x-b)\big)R'(x-b)^2 -R'\big(2R(x)-R(x-b)\big)R''(x-b).
\end{equation}

Since $R$ is positive, decreasing and strictly concave, it follows 
that (\ref{h1fact}) is positive and (\ref{h2fact}) is negative, hence that $h$ is increasing and strictly concave. 

So if $f'(x)\leq 1$,
giving $h'(0)\leq 1$, then $h(u)<u$ for all $u>0$. 
In that case the only non-negative fixed point of $h$
is $0$, and we must obtain $b=0$. In that case
the distribution $\mu^{(2)}$ must be a diagonal 
distribution, and we have bivariate uniqueness
(and hence endogeny).

On the other hand, suppose that $f'(x)>1$, 
so that $h'(0)>1$. Then for sufficiently small $\epsilon>0$,
$h(\epsilon)>\epsilon$. Starting from some such 
$\epsilon$ and iterating $h$ repeatedly
gives an increasing sequence which is bounded above
by $\min(x,1-x)$. Its limit is a fixed point of $h$
which lies in $(0,\min(x,1-x)]$. Hence in 
this case there does exist a non-degenerate bivariate 
fixed point, and the process is non-endogenous, as
required.
\end{proof}

\begin{proof}[Proof of Corollary \ref{cor:endogeny}]
Since $f'(x)=1$ everywhere, Theorem \ref{thm:endogeny} tells us
that all the processes with Bernoulli marginals are endogenous. 
This implies that for any $\mu$, for the process with marginals $\mu$, the event $\{Y\leq y\}$ is measurable with respect to the structure of the tree, for any $y$, where $Y$ is the value at the root. But then in fact the random variable $Y$ is measurable with respect to the structure of the tree, as required.
\end{proof}
 
\section*{Acknowledgments}
We thank Alexander Holroyd and Julien Berestycki for
valuable discussions, and Christina Goldschmidt and
Micha{\l} Przykucki for many helpful comments.
\bibliography{minimax}
\bibliographystyle{abbrv}

\bigskip

\noindent
\textsc{James B.\ Martin, Roman Stasi{\'n}ski}
\\
\textsc{Department of Statistics, University of Oxford, UK}
\\
\textit{E-mail addresses:}
\texttt{martin@stats.ox.ac.uk, roman.stasinski@stats.ox.ac.uk}
\\
\textit{URL:}
\texttt{http://www.stats.ox.ac.uk/\textasciitilde{}martin}

\end{document}